\documentclass[onefignum,onetabnum]{siamart190516}

\usepackage{amsfonts}
\usepackage{graphicx}
\usepackage{epstopdf}
\usepackage{algorithmic}
\usepackage{pdflscape}
\usepackage{afterpage}
\usepackage{xr}
\usepackage{enumitem}
\usepackage{tikz}

\usepackage{caption}
\usepackage{subcaption}

\usepackage{todonotes}

\usepackage{enumitem}

\setitemize[1]{itemsep=10pt,topsep=10pt}

\ifpdf
  \DeclareGraphicsExtensions{.eps,.pdf,.png,.jpg}
\else
  \DeclareGraphicsExtensions{.eps}
\fi

\newcommand{\ignore}[1]{}

\newcommand{\bs}[1]{\boldsymbol{#1}}

\newcommand{\R}{{\mathbb R}}

\newcommand{\Q}{{\mathbb P}}
\newcommand{\Z}{{\mathbf Z}}
\newcommand{\ES}{{\mathbf S}}

\newcommand{\CC}{\mathcal C}

\newcommand{\Fix}{\mathrm{Fix}}

\newcommand{\ONE}{{\mathbf 1}}
\newcommand{\Na}{{N_{\rm a}}}
\newcommand{\No}{{N_{\rm o}}}

\newcommand{\Zz}{\bs{Z}}

\newcommand{\Gg}{\bs{G}}

\newcommand{\Vc}{V_{\rm c}}
\newcommand{\Vd}{V_{\rm d}}
\newcommand{\Wdd}{V_{\rm dl}}
\newcommand{\Wd}{V_{\rm s}}

\newcommand{\cc}{c_{\rm c}}
\newcommand{\cd}{c_{\rm d}}
\newcommand{\cdd}{c_{\rm dl}}
\newcommand{\cdl}{c_{\rm s}}

\newcommand{\Vdl}{V_{\rm i}}
\newcommand{\Vp}{V_{\rm p}}

\newcommand{\talpha}{\tilde\alpha}
\newcommand{\tgamma}{\tilde\gamma}
\newcommand{\tbeta}{\tilde\beta}
\newcommand{\tdelta}{\tilde\delta}

\newcommand{\tcc}{\tilde c_{\rm c}}
\newcommand{\tcd}{\tilde c_{\rm d}}
\newcommand{\tcdd}{\tilde c_{\rm dl}}
\newcommand{\tcdl}{\tilde c_{\rm s}}

\newcommand{\beqn}{\begin{eqnarray*}}
\newcommand{\eeqn}{\end{eqnarray*}}

\newsiamremark{remark}{Remark}
\newsiamremark{hypothesis}{Hypothesis}
\crefname{hypothesis}{Hypothesis}{Hypotheses}
\newsiamthm{claim}{Claim}
\newsiamremark{warning}{Warning}
\newsiamremark{remarks}{Remarks}
\newsiamremark{example}{Example}

\newcommand{\Matrix}[1]{\ensuremath{\left[\begin{array}{ccccccccccccccccccccccccr} #1 \end{array}\right]}}
\newcommand{\Matrixr}[1]{\ensuremath{\left[\begin{array}{rrrrrrrrrrrrrrrrrrrrrrrrrrrrrrrrrrr} #1 \end{array}\right]}}

\usepackage{amsopn}

\newcommand{\GG}{{\mathcal G}}
 
 \newcommand{\DD}{\mathrm{D}}
 
 \newcommand{\KK}{{\mathcal K}}
 \newcommand{\id}{\mathrm{id}\,}

\usepackage{latexsym}
\usepackage{amsmath}
\usepackage{amssymb}

\headers{Decision from Indecision}{A.~Franci, M.~Golubitsky, I.~Stewart, A.~Bizyaeva, N.E.~Leonard}

\title{Breaking indecision \\ in multi-agent multi-option dynamics\thanks{Submitted \today.
\funding{This research has been supported in part by NSF grant CMMI-1635056, ONR grant N00014-19-1-2556, ARO grant W911NF-18-1-0325, DGAPA-UNAM PAPIIT grant IN102420, and Conacyt grant A1-S-10610. This material is also based upon work supported by the National Science Foundation Graduate Research Fellowship under Grant No. (NSF grant number).}}}

\author{Alessio Franci\thanks{Department of Mathematics, National Autonomous University of Mexico, 04510 Mexico City, Mexico. (\email{afranci@ciencias.unam.mx}, \url{https://sites.google.com/site/francialessioac/})}
\and Martin Golubitsky\thanks{Department of Mathematics, Ohio State University, Columbus, OH, 43210-1174 USA.
  (\email{golubitsky.4@osu.edu}).}
 \and Ian Stewart\thanks{Mathematics Institute, University of Warwick, Coventry, CV4 7AL, UK.\\
  (\email{i.n.stewart@warwick.ac.uk}).} 
\and Anastasia~Bizyaeva\thanks{Department 	of Mechanical and Aerospace Engineering, Princeton University, Princeton,
	NJ, 08544 USA. (\email{bizyaeva@princeton.edu}).}
\and Naomi Ehrich Leonard\thanks{Department
	of Mechanical and Aerospace Engineering, Princeton University, Princeton,
	NJ, 08544 USA.(\email{naomi@princeton.edu}).}
}
\date{\today}

\begin{document}

\maketitle

\begin{abstract}

How does a group of agents break indecision when deciding about options with qualities that are hard to distinguish? Biological and artificial multi-agent systems, from honeybees and bird flocks to bacteria, robots, and humans, often need to overcome indecision when choosing among options in situations in which the performance or even the survival of the group are at stake. Breaking indecision is also important because in a fully indecisive state agents are not biased toward any specific option and therefore the agent group is maximally sensitive and prone to adapt to inputs and changes in its environment. Here, we develop a mathematical theory to study how decisions arise from the breaking of indecision. Our approach is grounded in both equivariant and network bifurcation theory. We model decision from indecision as synchrony-breaking in influence networks in which each node is the value assigned by an agent to an option. First, we show that three universal decision behaviors, namely, deadlock, consensus, and dissensus, are the generic outcomes of synchrony-breaking bifurcations from a fully synchronous state of indecision in influence networks.
Second, we show that all deadlock and consensus value patterns and some dissensus value patterns are predicted by the symmetry of the influence networks. Third, we show that there are also many `exotic'  dissensus value patterns. These patterns are predicted by network architecture, but not by network symmetries, through a new {\it synchrony-breaking branching lemma}. This is the first example of exotic solutions in an application.
Numerical simulations of a novel influence network model illustrate our theoretical results.

\end{abstract}



\renewcommand{\Na}{m}
\renewcommand{\No}{n}

\section{Motivation}

Many multi-agent biological and artificial systems routinely make collective decisions. That is, they make decisions about a set of possible options through group interactions and without a central ruler or a predetermined hierarchy between the agents. Often they do so without clear evidence about which options are better. In other words, they make decisions from indecision. Failure to do so can have detrimental consequences for the group performance, fitness, or even survival.

Honeybee swarms recurrently look for a new home and make a selection when a quorum is reached~\cite[Seeley]{S97},\cite[Visscher]{K07}. Scouting bees evaluate the quality of the candidate nest site and recruit other bees at the swarm through the `waggle dance'~\cite[Seeley {\it et al.}]{SCS91}. It was shown in~\cite[Seeley {\it et al.}]{SKSHFM12} that honeybees use a cross-inhibitory signal and this allows them to break indecision when a pair of nest sites have near-equal quality (see also  \cite[Gray {\it et al.}]{GFSL18}, \cite[Pais {\it et al.}]{PHSFLM13}).

Animal groups on the move, like fish 
schools~\cite[Couzin {\it et al.}]{CIDGTHCLL11} or bird flocks ~\cite[Biro {\it et al.}]{BSMG06}, face similar conundrums when making navigational decisions, for instance, when there are navigational options with similar qualities or when the group is divided in their information about the options. However, when the navigational options appear sufficiently well separated in space, the group breaks indecision through social influence and makes a consensus 
navigational decision \cite[Couzin {\it et al.}]{CKFL05}, 
\cite[Leonard {\it et al.}]{LSNSCL12}, \cite[Sridhar {\it et al.}]{SLGNSBSGC21},\cite[Nabet {\it et al.}]{NLCL09}.

In the two biological examples above, the multi-agent system needs to make a {\it consensus} decision to adapt to the environment. But there are situations in which different members of the same biological swarm choose different options to increase the group fitness; that is, the group makes a {\it dissensus} decision.

This is the case of phenotypic differentiation in communities of isogenic social unicellular organisms, like 
bacteria \cite[Miller and Bassler]{MB01}, \cite[Waters and Bassler]{WB05}. In response to environmental cues, like starvation, temperature changes, or the presence of molecules in the environment, bacterial communities are able to differentiate into multiple cell types through molecular quorum sensing mechanisms mediated by autoinducers, as in {\it Bacillus subtilis} \cite[Shank and Kolter]{SK11} and {\it Myxococcus xanthus} \cite[Dworkin and Kaiser]{DK85} communities. In all 
those cases, the indecision about which cells should differentiate into which functional phenotype is broken through 
molecular social influence. In some models of sympatric speciation, which can be considered
as special cases of the networks studied here, coarse-grained sets of organisms
differentiate phenotypically by evolving different strategies \cite[Cohen and Stewart]{CS00}, 
\cite[Stewart {\it et al.}]{SEC03}. See Section~\ref{S:SCB}.

Artificial multi-agent systems, such as robot swarms, must make the same type of consensus and dissensus decisions as those faced by their biological counterparts. Two fundamental collective behaviors of a robot swarm are indeed collective decision making, in the form of either consensus achievement or task allocation, and coordinated navigation \cite[Brambilla {\it et al.}]{BFBD13}. Task allocation in a robot swarm is a form of dissensus decision making \cite[Franci {\it et al.}]{FBPL21} akin to phenotypic differentiation in bacterial communities. The idea of taking inspiration from biology to design robot swarm behaviors has a long history \cite[Kriger {\it et al.}]{KBK00},  \cite[Labella {\it et al.}]{LDD06}.

Human groups are also often faced with decision making among near equal-quality alternatives and need to avoid indecision, from deciding what to eat to deciding on policies that affect climate change~\cite{UN21}.

The mathematical modeling of opinion formation through social influence in non-hierarchical groups dates back to 
the De\-Groot linear averaging model \cite[DeGroot]{D74} and has since then evolved in a myriad of different models to reproduce different types of opinion formation behaviors (see for example \cite[Friedkin and Johnsen]{FJ99}, 
\cite[Deffuant {\it et al.}]{DNAW00}, and \cite[Hegselmann and Krause]{HK02}), including polarization (see for example 
\cite[Cisneros {\it et al.}]{CCB19}, \cite[Dandekar {\it et al.}]{DGL13}), and formation of beliefs on logically interdependent topics \cite[Friedkin {\it et al.}]{FPTP16},\cite[Parsegov {\it et al.}]{PPTF16},\cite[Ye {\it et al.}]{YTLAA20}. Inspired by the model-independent theory 
in~\cite[Franci {\it et al.}]{FGBL20}
we recently introduced a general model of nonlinear opinion dynamics to understand the emergence of, and the flexible transitions between, different types of agreement and disagreement opinion formation behaviors 
\cite[Bizyaeva {\it et al.}]{BFL22}.

In this paper we ask: Is there a general mathematical framework, underlying many specific models,
 for analyzing decision from indecision? We propose one answer based on differential 
 equations with network structure, and apply methods from symmetric dynamics, network dynamics,
 and bifurcation theory to deduce some general principles and facilitate the analysis
 of specific models.

The state space for such models is a rectangular array of agent-assigned option values, a `decision state', introduced in Section~\ref{S:influ_net}. Three important types of decision states are introduced: consensus, deadlock, dissensus.
Section~\ref{S:SBB} describes the bifurcations that can lead to the three types of decision states through indecision (synchrony) breaking from fully synchronous equilibria.
The decision-theoretic notion of an influence network, and of value dynamics over an influence network, together with the associated admissible maps  defining differential equations
that respect the network structure, are described in Section~\ref{S:network}, as are their symmetry groups and irreducible representations of those groups.  Section~\ref{S:PVD} briefly describes the interpretation of a distinguished parameter in the value dynamics equations from a decision-theoretic perspective. 
The trivial fully synchronous solution and the linearized admissible maps are discussed in Section~\ref{S:Deadlock_linearization}. 
It is noteworthy that the four distinct eigenvalues of the linearization are determined by symmetry.  This leads 
 to the three different kinds of synchrony-breaking steady-state bifurcation in the influence network. Consensus and deadlock synchrony-breaking bifurcation are studied in Section~\ref{S:DLB}.  Both of these bifurcations are 
 $\ES_N$ bifurcations that have been well studied in the symmetry-breaking context and are identical to the 
corresponding synchrony-breaking bifurcations. Section~\ref{S:axial} proves the Synchrony-Breaking Branching Lemma 
(Theorem~\ref{T:simp_eigen_reg}) showing the generic existence of solution branches for all `axial' balanced colorings, a key concept in homogeneous networks that is defined later.
This theorem is used in Section~\ref{S:diss_bif} to study dissensus synchrony-breaking bifurcations and prove the generic existence of `exotic' solutions that can be found using synchrony-breaking techniques but not using symmetry-breaking techniques. Section~\ref{S:SBSB} provides further discussion of exotic states. 
Simulations showing that axial solutions can be stable (though not necessarily near bifurcation) are
presented in Sections~\ref{S:simulation} and \ref{S:simulation_dissensus}. Additional discussion
of stability and instability of states is given in Sections~\ref{SS:stability_dissensus} and \ref{SS:stability_balanced}.

\section{Decision making through influence networks and value pattern formation}
\label{S:influ_net}

We consider a set of $\Na\geq 2$ identical agents who form valued preferences  
about a set of $\No\geq 2$ identical options. By `identical' we mean that
all agents process in the same way the valuations made by other agents,
and all options are treated in the same manner by all agents. These conditions are
formalized in Section~\ref{S:network} in terms of the network structure and symmetries of model equations.

When the option values are {\it a priori} 
unclear, ambiguous, or simply unknown, agents must compare and evaluate the 
various alternatives, both on their own and jointly with the other agents. They do 
so through an {\em influence network}, whose connections describe
mutual influences between valuations made by all agents.
We assume that
the value assigned by a given agent to a given option evolves over time, depending
on how all agents value all options. This dynamic is modeled as
a system of ordinary differential equations (ODE) whose formulas reflect
the network structure. Technically, this ODE is `admissible' \cite[Stewart {\it et al.}]{SGP03} for the network. 
See \eqref{EQ: generic value evol} and \eqref{EQ: admissible vf}.

In this paper we analyze the structure (and stability, when possible) of model ODE steady states.
Although agents and options are identical,
options need not be valued equally by all agents at steady state.
The uniform synchronous state, where all options are equally valued, can lose stability. This leads to spontaneous synchrony-breaking and the creation of patterns of values. 
Our aim is to exhibit a class of patterns, which we call `axial,' that can be proved to occur 
via bifurcation from a fully synchronous state for typical models.

\subsection{Synchrony-breaking versus symmetry-breaking}
\label{S:SBSB}

Synchrony-breaking is an approach to classify the bifurcating steady-state solutions of dynamical systems on homogeneous networks. It is analogous to the successful approach of spontaneous symmetry-break\-ing \cite[Golubitsky and Stewart]{GS02}, \cite[Golubitsky {\it et al.}]{GSS88}.  In its 
simplest form, symmetry-breaking addresses the following question. Given a symmetry group $\Gamma$ acting 
on $\R^n$, a stable $\Gamma$-symmetric equilibrium $x_0$, and a parameter $\lambda$, what are the possible symmetries 
of steady states that bifurcate from $x_0$ when $x_0(\lambda)$ loses stability at $\lambda=\lambda_0$? A technique that 
partially answers this question is the Equivariant Branching Lemma (see~\cite[XIII, Theorem~3.3]{GSS88}).  This 
result was first observed by \cite[Vanderbauwhede]{V82} and \cite[Cicogna]{Ci81}. The Equivariant Branching Lemma proves the existence of a 
branch of equilibria for each axial subgroup $\Sigma\subseteq\Gamma$. It applies whenever an axial subgroup exists.

Here we prove an analogous theorem for synchrony-breaking in the network context. Given a homogeneous 
network $\mathcal N$ with a fully synchronous stable equilibrium $x_0$ and a parameter $\lambda$, 
we ask what kinds of patterns of 
synchrony bifurcate from $x_0(\lambda)$ when the equilibrium loses stability at $\lambda_0$.  A technique that 
partially answers this question is the Synchrony-Breaking Branching Lemma  \cite[Golubitsky and Stewart]{GS22} (reproved here 
in Section~\ref{S:axial}). This theorem
proves the existence of a branch of equilibria for each `axial' subspace 
$\Delta_{\bowtie}$.
It applies whenever an axial subspace exists.

The decision theory models considered in this paper are both homogeneous (all nodes are of the same kind) and symmetric (nodes are interchangeable).  Moreover, the 
maximally symmetric states (states fixed by all symmetries) are the same as the fully synchronous states and 
it can be shown that for influence networks the same critical eigenspaces (where bifurcations occur)
are generic in either context; namely, the absolutely irreducible representations. 
However, the generic branching behavior on the critical eigenspaces
can be different for network-admissible ODEs compared to equivariant ODEs. This occurs because network structure imposes extra constraints compared to symmetry alone and therefore network-admissible ODEs are in general a subset of equivariant ODEs. Here 
it turns out that this difference occurs for dissensus bifurcations but not for deadlock or consensus bifurcations.
Every axial subgroup $\Sigma$ leads to an axial subspace $\Delta_{\bowtie}$, but the converse need not be 
true.  We call axial states that can be obtained by symmetry-breaking {\em orbital states} and axial states that 
can be obtained by the Synchrony-Breaking Branching Lemma, but not by the Equivariant Branching Lemma, 
{\em exotic states}.  We show that dissensus axial exotic states often exist when $m,n\geq 4$. These states 
are the first examples of this phenomenon that are known to appear in an application.

\subsection{Terminology}

The {\em value} that agent $i$ assigns to option $j$ is represented by a real number $z_{ij}$
which may be positive, negative or zero. 
A {\em state} of the influence network is a rectangular array of values $Z=(z_{ij})$
where $1 \leq i \leq \Na$ and $1 \leq j \leq \No$. 
The assumptions that agents/options
are identical, as explained at the start of this section, implies
that any pair of entries $z_{ij}$ and $z_{kl}$ can be compared and that this comparison is meaningful.

{\bf Decision states}: A state where each row consists of equal values is a {\em deadlock} state (Fig~\ref{fig:consensus illust}, right). Agent $i$ is 
{\em deadlocked} if row $i$ consists of equal values. A state where each column 
consists of equal values is a {\em consensus} state (Fig~\ref{fig:consensus illust}, left).  There is {\em consensus about option} $j$ if column $j$ 
consists of equal values. A state that is neither deadlock nor consensus is {\em dissensus} (Fig.~\ref{fig:dissensus illust ab}). A state that is both deadlock and consensus is a state of {\it indecision}. An indecision state is {\it fully synchronous}, that is, in such a state all values are equal.

{\bf Value patterns}: We discuss value patterns that are likely to form by synchrony-breaking 
bifurcation from a fully synchronous steady state.   Each bifurcating state forms a pattern of value 
assignments. It is convenient to visualize this pattern by assigning a color
to each numerical value that occurs. Nodes $z_{ij}$ and $z_{kl}$ of the array $Z$
have the same color if and only if $z_{ij} = z_{kl}$.  Nodes with the same value, hence same color, are {\em synchronous}.  Examples of value patterns are shown in Figures~\ref{fig:consensus illust}-\ref{fig:dissensus illust c}. Agent and option numbering are omitted when not expressly needed.

%

\subsection{Sample bifurcation results}

Because agents and options are independently interchangeable,
 the influence network has 
symmetry group $\Gamma = \ES_m\times\ES_n$ and admissible systems of ODEs are  
$\Gamma$-equivariant.  Equivariant bifurcation theory shows that admissible systems can exhibit 
three kinds of symmetry-breaking or synchrony-breaking bifurcation from a fully synchronous equilibrium;
see Section~\ref{SS:3symm_breaking}.  Moreover, axial patterns associated with these bifurcation types lead 
respectively to consensus (\S\ref{case1}, Theorem~\ref{T:consensus_bifur}), 
deadlock (\S\ref{case2}, Theorem~\ref{T:divided_deadlock_bifur}),
or dissensus (\S\ref{case3}, Theorem~\ref{T:diss}) states. 
Representative examples (up to renumbering agents and options) are shown
 in Figures~\ref{fig:consensus illust} and \ref{fig:dissensus illust ab}.

\begin{figure}[htb]
	\centering
	\includegraphics[width=0.95\textwidth]{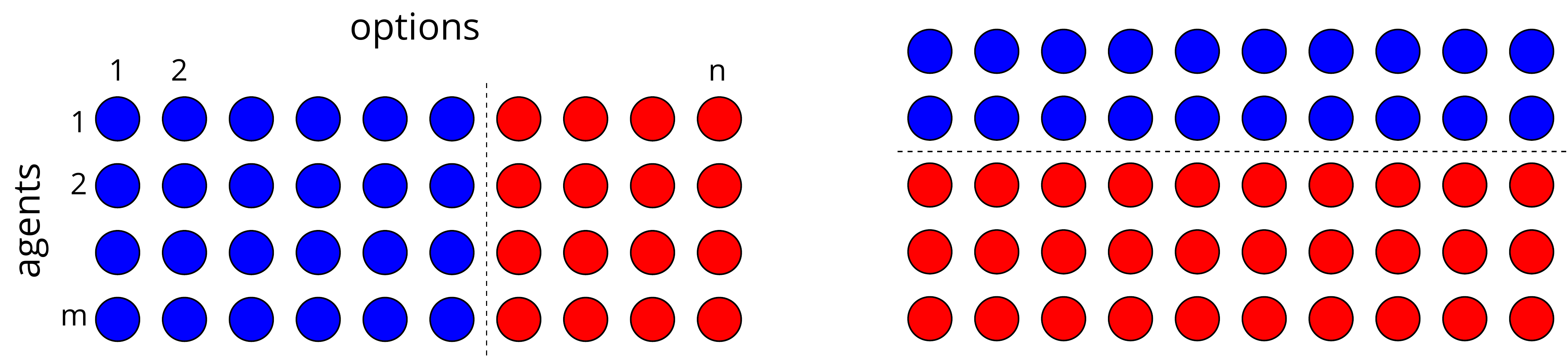}
	\caption{(Left) Consensus value pattern. (Right) Deadlock value pattern.}\label{fig:consensus illust}
\end{figure}

\begin{figure}[htb]
	\centering
	\includegraphics[width=0.9\textwidth]{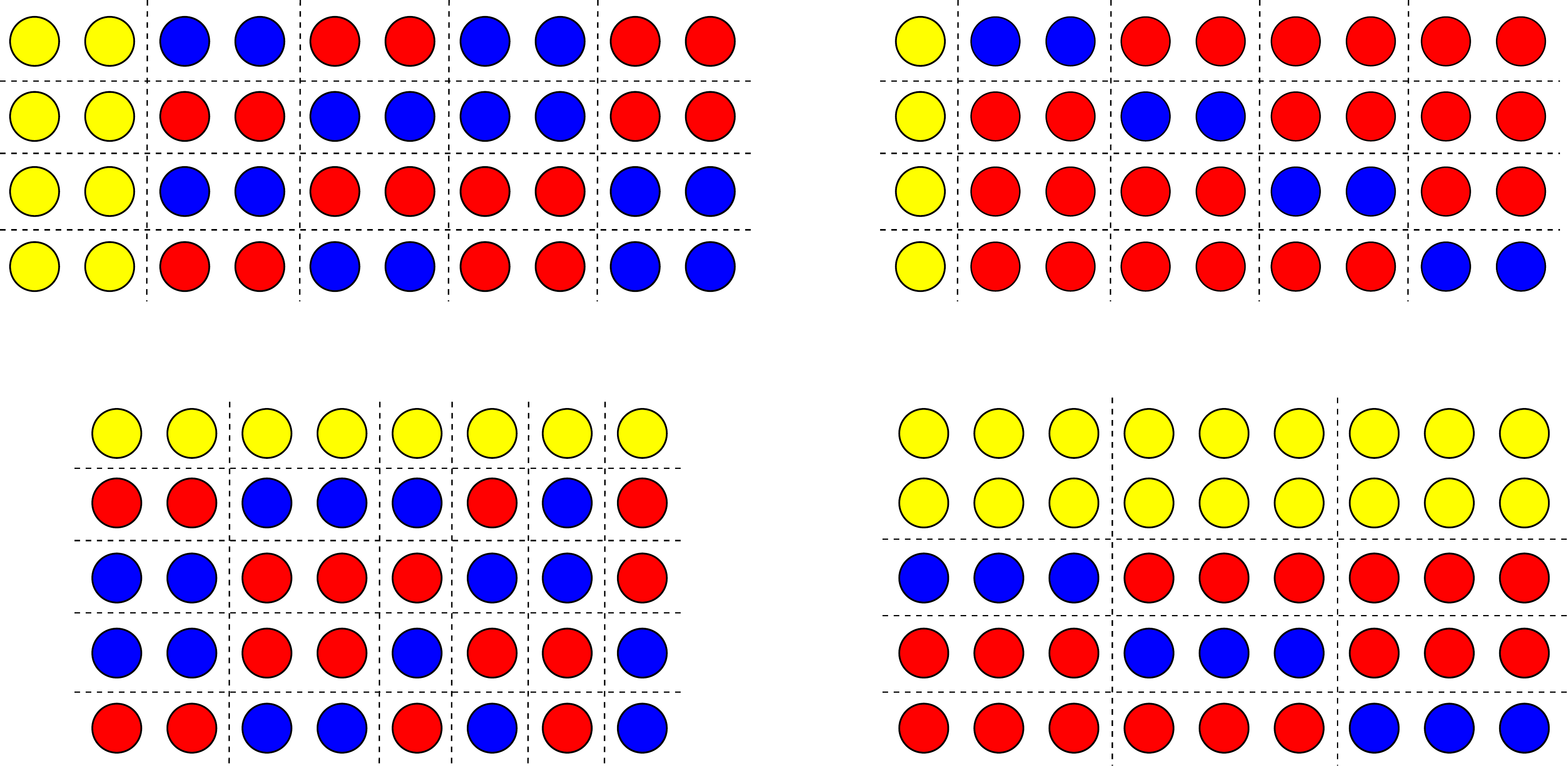}
	\caption{Dissensus value patterns corresponding to \S\ref{case3a} (top) and \S\ref{case3b} (bottom). }\label{fig:dissensus illust ab}
\end{figure}

The different types of pattern are characterized by the following properties: zero row-sums, zero column-sums, equality of rows, and equality of columns.
	
A bifurcation branch has {\em zero row-sum} (ZRS) if along the branch each row sum is constant to linear order 
in the bifurcation parameter. 
A bifurcation branch has {\em zero column-sum} (ZCS) if along the branch each column sum is constant to linear order 
in the bifurcation parameter. 

Thus, along a ZRS branch, an agent places higher values on some options and lower values on other options, but in such a way that the average value remains constant. In particular, along a ZRS branch, an agent forms preferences for the various options, e.g., favoring some and disfavoring others. Similarly, along a ZCS branch, an option is valued higher by some agents and lower by some others, but in such a way that the average value remains approximately constant. In particular, 
along a ZCS branch, an option is favored by some agents and disfavored by others.

Consensus  bifurcation branches are ZRS with all rows equal, so all agents favor/disfavor the same options. Deadlock bifurcation branches are ZCS with all columns equal, so each agent equally favors or disfavors every option. Dissensus bifurcation branches are ZRS and ZCS with some unequal rows and some unequal columns, so there is no consensus on favored/disfavored options.

ZRS and ZCS conditions  are discussed in more detail
in Section~\ref{SSEC: irreps} using the irreducible representations of $\Gamma$.

\subsection{Value-assignment clusters}

To describe these patterns informally in decision theoretic language, we first define:

\begin{definition} \rm
	Two agents are in the same {\em agent-cluster} if their rows are equal. Two options are in the same {\em option-cluster} if their columns are equal.
\end{definition}

In other words, two agents in the same agent-cluster agree on each option value and two options 
are in the same option-cluster if they are equally valued by each agent.
An example is given in Figure~\ref{F:color_example_rect}.

\begin{figure}[htb]
	\centerline{
		\includegraphics[width=.27\textwidth]{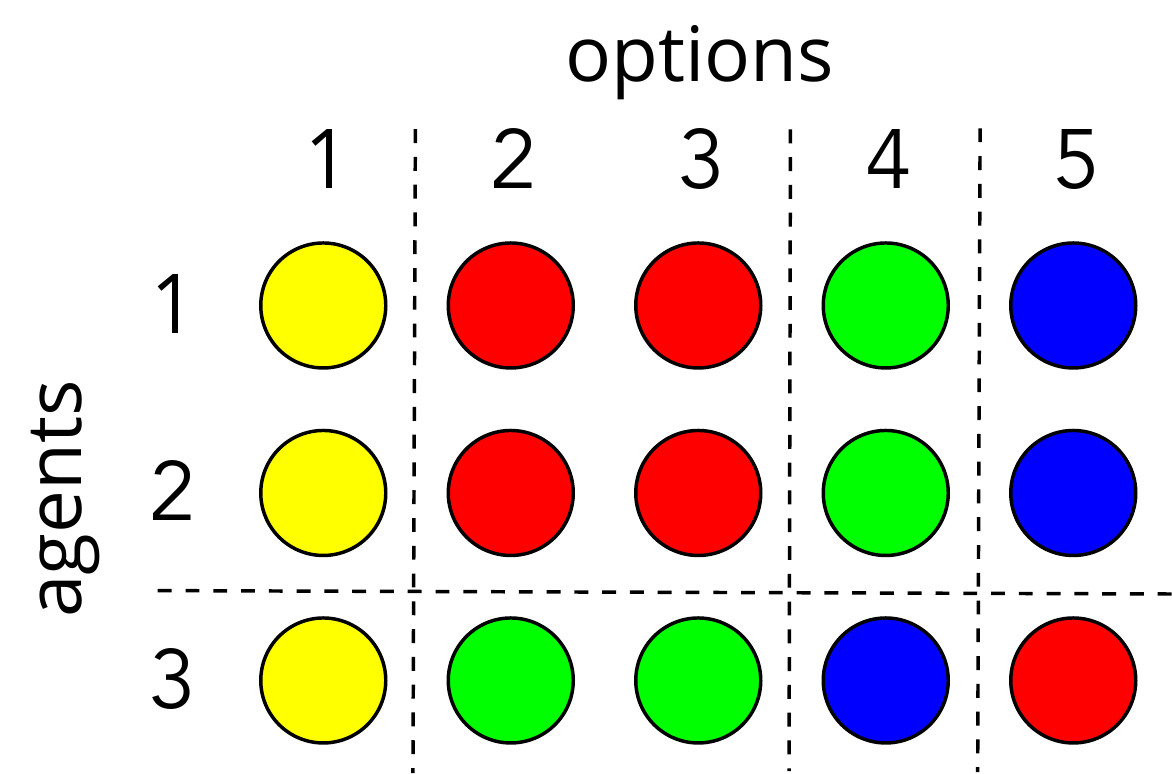}}
\caption{The agent-clusters in this value pattern are $\{1,2\}$, $\{3\}$ and the 
option-clusters are $\{1\}, \{2,3\}, \{4\}, \{5\}$.}
	\label{F:color_example_rect}
\end{figure}

\subsection{Color-isomorphism and color-complementarity}

The following definitions help characterize value patterns. Definition~\ref{D: color isom} (color-isomorphism) relates two agents by an option permutation and two options by an agent permutation, thus preserving the number of nodes of a given color in the associated rows/columns. Definition~\ref{D: color comp} (color-complementarity) relates two agents/options by a color permutation, therefore not necessarily preserving the number of nodes of a given color in the associated 
rows/columns.

\begin{definition}\label{D: color isom} \rm
	Two agents (options) are {\em color-isomorphic} if the row (column) pattern of one can be transformed into the row (column) pattern of the other by option (agent) permutation.
\end{definition}

\begin{figure}
	\centering
	\includegraphics[width=0.6\textwidth]{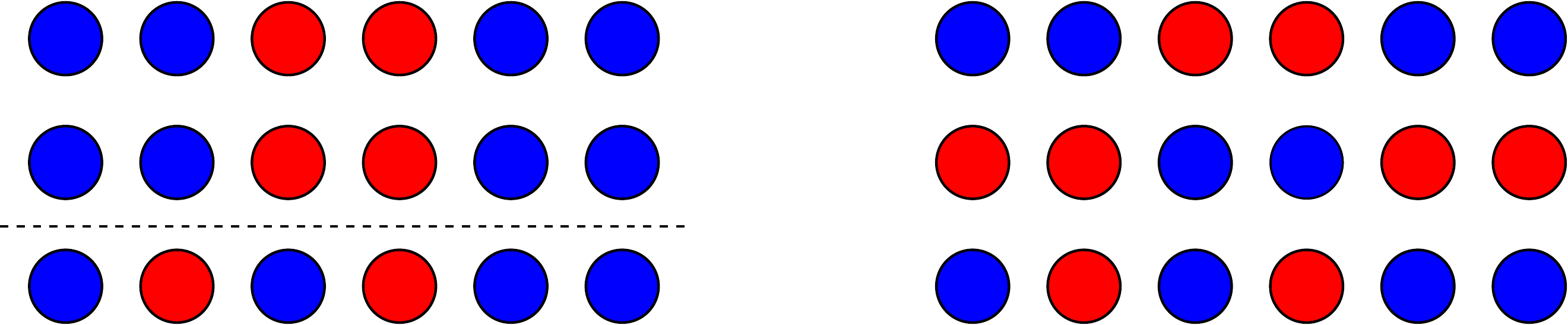}
	\caption{	(Left) \{1, 2\}, \{3\} are agent clusters; agents 1, 2, 3 are color-isomorphic;  none of the agents are color complementary. (Right) \{1\}, \{2\},\{3\} are agent clusters; agents 1 and 3  are color isomorphic; agents 1 and 2 are color complementary.}
	\label{fig:cluster vs isom vs compl}
\end{figure}

\begin{definition}\label{D: color comp}\rm 
 Two agents (options) are {\em color-complementary} if the row (column) pattern of one can be transformed into the row (column) pattern of the other by a color permutation.
\end{definition}

Examples of agent-clusters, color-isomorphism, and color-complementarity are given 
in Figure~\ref{fig:cluster vs isom vs compl}.

Non-trivial color-isomorphism and color-complementarity characterize {\it disagreement} value patterns. Two color-isomorphic agents not belonging to the same agent-cluster disagree on the value of the permuted options. Two color-isomorphic options not belonging to the same option-cluster are such that the permuted agents disagree on their value. Similarly, two color-complementary agents not belonging to the same agent-cluster disagree on options whose colors are permuted. Two color-complementary options not belonging to the same option-cluster are such that the agents whose colors were permuted disagree on those option values.

\section{Indecision-breaking as synchrony breaking from full synchrony}
\label{S:SBB}

We study patterns of values in the context of
{\em bifurcations}, in which the solutions of a parametrized family of ODEs change qualitatively as a parameter $\lambda$ varies. Specifically, we consider
{\em synchrony-breaking} bifurcations, in which a branch of fully synchronous
steady states ($z_{ij} = z_{kl}$ for all $i,j,k,l$) becomes dynamically unstable,
leading to a branch (or branches) of states with a more complex value pattern. The fully 
synchronous state is one in which the group of agents expresses 
no preferences about the options. In such a state it is impossible for the single agent and for the group as a whole to coherently decide which options are best. The fully synchronous state therefore correspond to indecision; an agent group in such a state is called {\it undecided}.
	
Characterizing indecision is important because any possible value pattern configuration can be realized as an infinitesimal perturbation of the fully synchronous state. In other words, any decision can be made from indecision. Here, we show how exchanging opinions about option values through an influence network leads to generic paths from indecision to decision through synchrony-breaking bifurcations.


The condition for steady-state bifurcation is that the linearization $J$ 
of the ODE at the bifurcation
point has at least one eigenvalue equal to $0$. The corresponding eigenspace is
called the {\em critical eigenspace}.
In~\eqref{Irrep_decomp} we show that in the present
context there are three possible critical eigenspaces for generic synchrony-breaking.
The bifurcating steady states are described in Theorems~\ref{T:consensus_bifur}, \ref{T:divided_deadlock_bifur}, \ref{T:diss}. 
For each of these cases, bifurcation theory for networks leads to a list of `axial'
patterns --  for each we give a decision-making theoretic interpretation in Sections~\ref{case1}--\ref{case3}.
These informal descriptions are made precise by the mathematical statements
of the main results in Sections~\ref{S:DLB} and \ref{S:diss_bif}.

In the value patterns in Figures~\ref{fig:consensus illust},~\ref{fig:dissensus illust ab},~\ref{fig:dissensus illust c}, we use the convention that the value of red nodes is small, the value of blue nodes is high, and the value of yellow nodes is between that of red and blue nodes. Blue nodes correspond to favored options, red nodes to disfavored options, and yellow nodes to options about which an agent remains neutral. That is, to 
linear order yellow nodes assign to that option the same value as before synchrony is broken.
Color saturation of red and blue encodes for deviation from neutrality, where appropriate.

\subsection{Consensus bifurcation}
\label{case1}

All agents assign the same value to any given option. There is one
agent-cluster and two option-clusters. All values for a given option-cluster
are equal and different from the values of the other cluster. Options from different option-clusters are not color-isomorphic but are color-complementary. One option cluster is made of favored options and the other option cluster of disfavored options. In other words, there is a group consensus about which options are favored. See Figure~\ref{fig:consensus illust}, left, for a typical consensus value pattern.  The sum of the $z_{ij}$ along each row is zero to leading order in $\lambda$.

\subsection{Deadlock bifurcation}
\label{case2}

All options are assigned the same value by any given agent. There is one
option-cluster and two agent-clusters. All values for a given agent-cluster
are equal and different from the values of the other cluster. Agents from different agent-clusters are not color-isomorphic but are color-complementary. One agent cluster favors all the options and the other agent cluster disfavor all options. Each agent is deadlocked about the options and the group is divided about the overall options' value. See Figure~\ref{fig:consensus illust}, right, for a typical deadlock value pattern.  The sum of the $z_{ij}$ along each column is zero to leading order in $\lambda$.

\subsection{Dissensus bifurcation}
\label{case3}

This case is more complex, and splits into three subcases.

\subsubsection{Dissensus with color-isomorphic agents}
\label{case3a}

Each agent assigns the same high value to a subset of favored options, the same low value to a subset of disfavored options, and, possibly, remains neutral about a third subset of options.  All agents are color-isomorphic but there are at least two distinct agent-clusters. Therefore, there is no consensus about which options are (equally) favored.

There are at least two distinct option-clusters with color-isomorphic elements. Each agent expresses preferences about the elements of each option cluster by assigning the same high value to some of them and the same low value to some others. There might be a further option-cluster made of options about which all agents remain neutral.

There might be pairs of color-complementary agent and option clusters.  See Figure~\ref{fig:dissensus illust ab}, top for typical dissensus value patterns of this kind with (left) and without (right) color-complementary pairs.

\subsubsection{Dissensus with color-isomorphic options}
\label{case3b}

This situation is similar to \S\ref{case3a}, but now all options are color-isomorphic and there are no options about which all agents remain neutral. There might be an agent-cluster of neutral agents, who assign the same mid-range value to all options. See Figure~\ref{fig:dissensus illust ab}, bottom for typical dissensus value patterns of this kind with (left) and without (right) color-complementary pairs.

\subsubsection{Polarized dissensus}
\label{case3c}

There are two option-clusters and two agent-clus\-ters. For each combination of one option-cluster and one agent-cluster, all values are identical. Members of distinct agent-clusters are color-isomorphic and color-comple\-ment\-ary if and only if  members of distinct option-clusters are color-isomorphic and color-complementary. Equivalently, if and only if the number of elements in each agent-cluster is the same and the number of elements in each option-cluster is the same.

In other words, agents are polarized into two clusters. Each cluster has complete consensus on the option values, and places the same value on all options within a given option-cluster. The second agent-cluster does the same, but disagrees with the first agent-cluster on all values. See Figure~\ref{fig:dissensus illust c} for typical polarized dissensus value patterns without (left) and with (right) color-isomorphism and color-complementarity.

\begin{figure}
	\centering
	\includegraphics[width=0.9\textwidth]{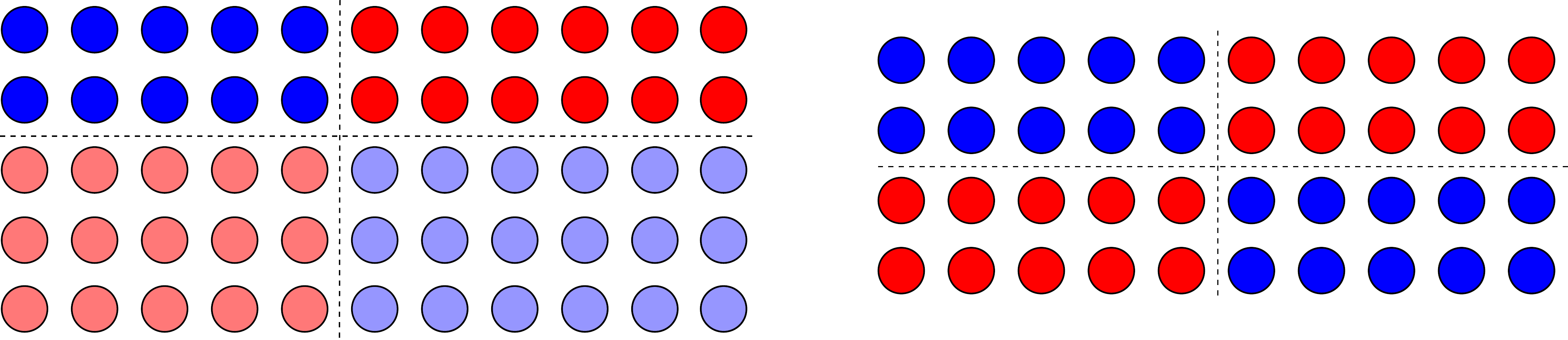}
	\caption{Examples of polarized dissensus value patterns (\S\ref{case3c}).}\label{fig:dissensus illust c}
\end{figure}

\section{Influence Networks}
\label{S:network}

We now describe the mathematical context in more detail and make the previous
informal descriptions precise.

The main modeling assumption is that
 the value $z_{ij}\in\R$ assigned by agent $i$ to option $j$ evolves continuously in time. 
We focus on dynamics that evolve towards equilibrium, and describe generic
 equilibrium synchrony patterns.
  Each variable $z_{ij}$ is associated to a {node $(i,j)$} in an {\it influence network}. 
 The set of nodes of the influence network is the rectangular array 
\[
\CC=\{(i,j):1\leq i \leq \Na,1\leq  j \leq \No\}.
\]
{We assume that all} nodes are {identical;} they are represented by the same symbol {in} Figure~\ref{FIG: influence network} (left). {Arrows} between the nodes {indicate influence; that is,} an arrow from node $(k,l)$ to node $(i,j)$ means that the value assigned by agent $i$ to {option} $j$ is influenced by the value assigned by agent $k$ to option $l$.

\subsection{Arrow types}

We assume that there are three different types of interaction between nodes,
determined by different types of arrow in the network, represented graphically
by different styles of arrow, Figure~\ref{FIG: influence network} (left). 
The arrow types are distinguished by their tails:
\begin{itemize}
	
	\item (Row arrows) The {\it intra-agent, inter-option influence neighbors} of node $(i,j)$ are the $\No - 1$ 
	nodes in the subset 
	\[
	{\sf A}_{i,j} = \{(i,1),\ldots, (i,\No)\} \setminus \{(i,j)\}.
	\]
	Arrows whose head is $(i,j)$ and whose tail is in ${\sf A}_{i,j}$ have the same arrow type represented by 
	a gray dashed arrow. In Figure~\ref{FIG: influence network} (right) these arrows connect all distinct nodes in the same 
	row in an all-to-all manner.
	
	\item (Column arrows) The {\it inter-agent, intra-option influence} neighbors of node $(i,j)$ are the $\Na - 1$ nodes in the subset 
	\[
	{\sf O}_{ij} = \{(1,j),\ldots, (\Na, j)\} \setminus \{(i,j)\}.
	\]
	Arrows whose head is $(i,j)$ and whose tail is in ${\sf O}_{ij}$ have the same arrow type represented 
	by a solid black arrow. In Figure~\ref{FIG: influence network} (right) these arrows connect all distinct nodes in the same 
	column in an all-to-all manner.

	\item (Diagonal arrows) The {\it inter-agent, inter-option influence} neighbors of node $(i,j)$ are 
	the $\Na\No - \Na -\No + 1$ nodes in the subset 
	\[
	{\sf E}_{ij} = \{(k,l): k\neq i, l\neq j \}
	\]
	Arrows whose head is $(i,j)$ and whose tail is in ${\sf E}_{ij}$ have the same arrow type represented by  a 
	black dashed arrow.
	In Figure~\ref{FIG: influence network} (right) these arrows connect all distinct nodes that do not lie in the same row or in the same column in an all-to-all manner.
 \end{itemize}
 
 We denote an $m \times n$ influence network by $\mathcal{N}_{mn}$.
The network $\mathcal{N}_{mn}$ is {\em homogeneous} (all nodes have isomorphic
sets of input arrows) and {\em all-to-all coupled} or {\em all-to-all connected} (any two distinct nodes
appear as the head and tail, respectively, of some arrow).

Other assumptions on arrow or node types are possible, 
to reflect different modeling assumptions, but are not discussed in this paper.

\begin{figure}
	\centering
	\includegraphics[width=0.65\textwidth]{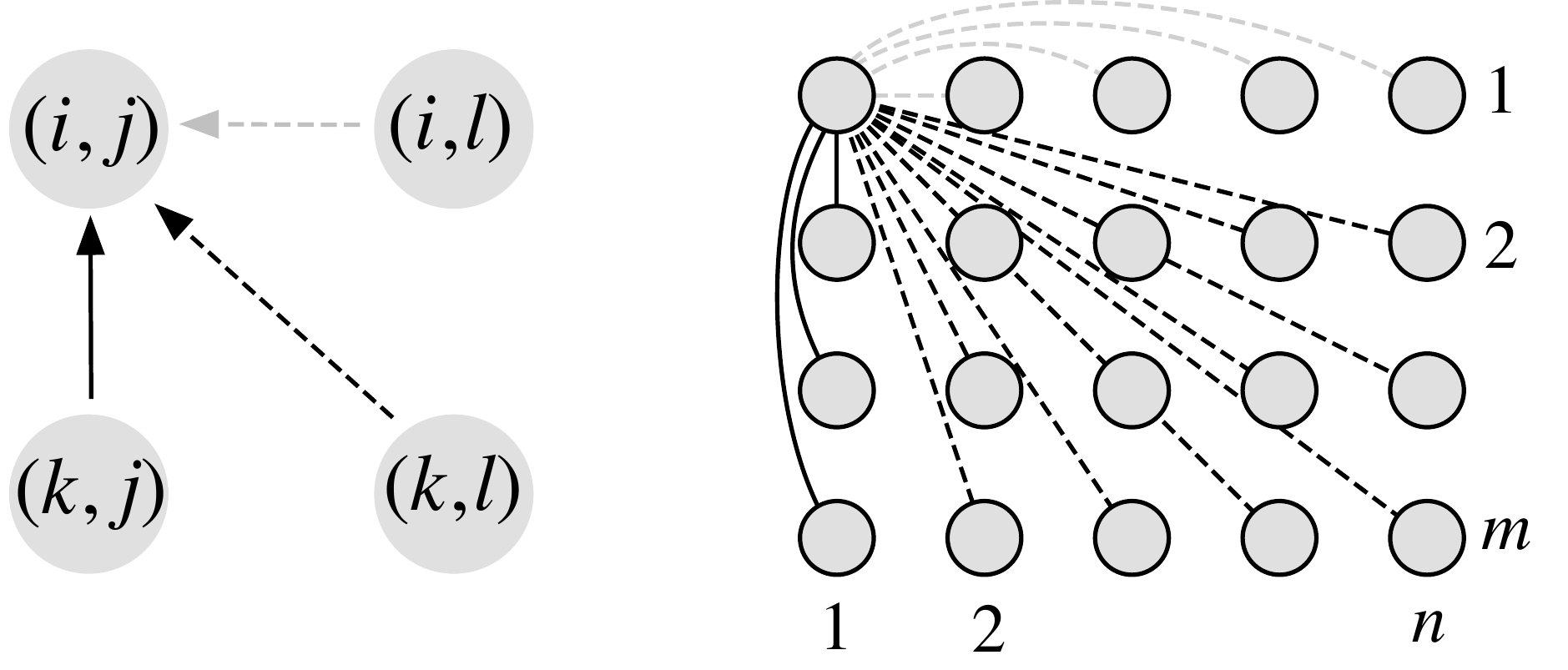}
	\caption{An $m \times n$ influence network $\mathcal{N}_{mn}$ has three distinct arrow types: gray dashed row arrow, solid black column arrow, and black dashed diagonal arrow. (Left) Inputs to
	node $(i,j)$. (Right) Arrows inputting to node $(1,1)$ (arrowheads omitted). There is a similar set of arrows inputting to every other node.}
	\label{FIG: influence network}
\end{figure}

\subsection{Symmetries of the influence network}
\label{S:SIN}

The influence network $\mathcal{N}_{mn}$ is symmetric: its {\em automorphism group}, the set of
 permutations of the set of nodes $\CC$ that preserve node and arrow types and incidence relations
 between nodes and arrows, is non-trivial. In particular, swapping any two rows
 or any two columns
  in Figure~\ref{FIG: influence network} (right) leaves the network structure unchanged. 
It is straightforward to prove that $\mathcal{N}_{mn}$ has symmetry group 
\[
\Gamma = \ES_\Na\times\ES_\No
\]
where $\ES_\Na$ swaps rows (agents) and $\ES_\No$ swaps columns (options). More precisely, 
$\sigma\in\ES_{\Na}$ and $\tau\in\ES_{\No}$ act on $\CC$ by
\[
(\sigma,\tau) (i,j) = (\sigma(i), \tau(j)).
\]

The symmetries of the influence network can be interpreted 
as follows. In the decision making process, agents and options are {\it a priori} 
indistinguishable. In other words, all the agents count the same and have the 
same influence on the decision making process, and all the options are initially equally 
valuable. This is the type of decision-making situation we wish to model and 
whose dynamics we wish to understand.

\subsection{Admissible maps}

A point in the {\em state space} (or {\em phase space}) $\Q$ is an $\Na\times\No$ rectangular array of real numbers $z = (z_{ij})$. 
The action of $\Gamma$ on points $z = (z_{ij})\in\Q$ is defined by
\[
(\sigma,\tau)z = ( z_{ \sigma^{-1}(i)\tau^{-1}(j)}). 
\]

Let 
\[
\Gg:\Q\to\Q
\]
be a smooth map $\Gg =(G_{ij})$ on $\Q$, so that each component
 $G_{ij}$ is smooth and real-valued.
We assume that the value $z_{ij}$ assigned by agent $i$  to option $j$ evolves 
according to the {\it value dynamics}
\begin{equation}\label{EQ: generic value evol}
\dot z_{ij} = G_{ij}(z).
\end{equation}
In other word, we model the evolution of the values assigned by the agents to the different options as a system of ordinary differential equations (ODE) on the state space $\Q$.

The map $\Gg$ for the influence network is assumed to be 
{\it admissible}~\cite[Definition~4.1]{SGP03}. 
Roughly speaking, admissibility means that the functional dependence of the map $\Gg$ on the network variable $z$ respects the network structure. For instance, if two nodes $(k_1,l_1),(k_2,l_2)$ input a third one $(i,j)$ through the same arrow type, then $\Gg_{ij}(z)$ depends identically on $z_{k_1l_1}$ and $z_{k_2l_2}$, in the sense that the value of $\Gg_{ij}(z)$ does not change if $z_{k_1l_1}$ and $z_{k_2l_2}$ are swapped in the vector $z$.
Our assumptions about
 the influence network select a family of admissible maps that can be analyzed from a network perspective to predict generic model-independent decision-making behaviors.

Because the influence network has symmetry $\ES_m\times\ES_n$ and there are three arrow types, by~\cite[Remark~4.2]{SGP03}, the components of the admissible
map $\Gg$ in~\eqref{EQ: generic value evol} satisfy
\begin{equation}\label{EQ: admissible vf}
G_{ij}(z)=G(z_{ij}, \overline{z_{\sf{A_{ij}}}}, \overline{z_{\sf{O_{ij}}}}, \overline{z_{\sf{E_{ij}}}})
\end{equation}
where the function $G$ is independent of $i,j$ and the notation $\overline{z_{\sf{A_{ij}}}}$, $\overline{z_{\sf{O_{ij}}}}$, and $\overline{z_{\sf{E_{ij}}}}$ means 
that $G$ is invariant under all permutations of the arguments appearing under each overbar.
That is, each arrow type leads to identical interactions.

It is well known~\cite{AS07,GS22} that the symmetry of $\mathcal{N}_{mn}$ implies that
any admissible map $\Gg$ is $\Gamma$-equivariant, that is,
\[
\gamma\Gg(z)=\Gg(\gamma z)
\]
for all $\gamma\in\Gamma$. It is straightforward to verify directly that equations 
 \eqref{EQ: admissible vf} are $\Gamma$-equivariant. 
 Given $\gamma=(\sigma,\tau)\in\Gamma$, with
$\sigma\in\ES_\Na$ and $\tau\in\ES_\No$,
\[
\begin{array}{rcl}
	((\sigma,\tau)\Gg)_{ij}(z)& = & \!G_{\sigma^{-1}(i)\, \tau^{-1}\!(j)}(z) \\ & &  \\
	 & = &  G(z_{\sigma^{-1}(i)\, \tau^{-1}(j)}, \overline{z_{\sf{A_{\sigma^{-1}(i)\,\tau^{-1}\!(j)}}}},
	 \overline{z_{\sf{O_{\sigma^{-1}(i)\,\tau^{-1}(j)}}}}, \overline{z_{\sf{E_{\sigma^{-1}(i)\,\tau^{-1}(j)}}}})\\
	 & & \\&= & G_{ij}((\sigma,\tau)z).
\end{array}
\]

\subsection{Solutions arise as group orbits}
\label{S:sol_orbits}
An important consequence of group equivariance is that solutions
of any $\Gamma$-equivariant ODE arise in group orbits. That is,
if $z(t)$ is a solution and $\gamma \in \Gamma$ then $\gamma z(t)$
is also a solution. Moreover, the type of solution (steady, periodic, chaotic)
is the same for both, and they have the same stabilities. Such solutions
are often said to be {\em conjugate} under $\Gamma$.

In an influence network, conjugacy has a simple interpretation. 
If some synchrony pattern can occur as an equilibrium, then so can all related patterns
in which agents and options are permuted in any manner. 
Effectively,
we can permute the numbers $1, \ldots, m$ of agents and $1, \ldots, n$ of
options without affecting the possible dynamics. This is reasonable
because the system's behavior ought not to depend on the choice of numbering.

Moreover, if one
of these states is stable, so are all the others. This phenomenon is a direct
consequence of the initial assumption that all agents are identical and all options 
are identical. Which state among the set of conjugate ones occurs, in
any specific case, depends on the initial conditions for the equation.
In particular, whenever we assert the existence of a pattern as an equilibrium, 
we are also asserting the existence of all conjugate patterns as equilibria.

\subsection{Equivariance versus admissibility}
 
If a network has symmetry group $\Gamma$
then every admissible map is $\Gamma$-equivariant, but the converse is
 generally false. In other words, admissible maps for the influence network are more constrained than $\Gamma$-equivariant maps. As a consequence, bifurcation phenomena that are not generic in equivariant maps (because not sufficiently constrained) become generic in admissible maps (because the extra constraints provide extra structure).\footnote{As an elementary analogue of this phenomenon, consider the class of real-valued maps $f(x,\lambda)$ and the class of odd-symmetric real-valued maps $g(x,\lambda)=-g(-x,\lambda)$. The first class clearly contains the second, which is more constrained (by symmetry). Whereas the pitchfork bifurcation is non-generic in the class of real valued maps, it becomes generic in the odd-symmetric class.}
 In practice, this means that equivariant bifurcation theory applied to influence networks might miss some of the bifurcating solutions, i.e., those that are generic in admissible maps but not in equivariant maps.
 Section~\ref{S:EP} illustrates one way in which this can happen. Other differences between bifurcations in equivariant and admissible maps are discussed in detail in~\cite[Golubitsky and Stewart]{GS22}.
 
 For influence networks, we have the following results: 
  
 \begin{theorem}
\label{T:equiv_not_admiss}
The equivariant maps for $\ES_m\times\ES_n$ are the same as the admissible maps
for ${\mathcal N}_{mn}$ if and only if $(m,n) = (m,1), (1,n), (2,2), (2,3), (3,2)$.
\end{theorem}
\begin{proof}
See~\cite[Theorem 3]{S21}.
\end{proof}

In contrast, the linear case is much better behaved:

\begin{theorem}
\label{T:ln_equiv_admiss}
The linear equivariant maps for $\ES_m\times\ES_n$ are the same as the linear
admissible maps
for ${\mathcal N}_{mn}$ for all $m,n$.
\end{theorem}
\begin{proof}
See~\eqref{distinct_irreps} and Remark~\ref{r:admiss=equi}.
\end{proof}

\subsection{Irreducible representations of $\Gamma$ acting on $\Q$}
\label{SSEC: irreps}

In equivariant bifurcation theory, a key role is played by the irreducible
representations of the symmetry group $\Gamma$ acting on $\Q$. 
Such a representation is {\em absolutely irreducible} if all commuting
linear maps are scalar multiples of the identity. See~\cite[Chapter XII Sections 2--3]{GSS88}.

The state space $\Q = \R^{\Na\No}$ admits a direct sum decomposition into four $\Gamma$-invariant subspaces  with a clear decision-making interpretation for each.  
These subspaces of $\Q$ are
\begin{equation} \label{distinct_irreps}
\begin{array}{ccl}
\Wd & = & \mbox{all entries equal} \quad  (\mbox{{\it fully synchronous subspace}})  \\
\Vc & = & \mbox{all rows identical with sum $0$} \quad (\mbox{{\it consensus subspace}}) \\
\Wdd & = & \mbox{all columns identical with sum $0$} \quad (\mbox{{\it deadlock subspace}}) \\
\Vd &=& \mbox{all rows and all columns have sum 0} \quad  (\mbox{{\it dissensus subspace}})\\
\end{array}
\end{equation}
whose dimensions are respectively $ 1, (\No -1)(\Na -1), \No-1, \Na -1$.  

Each subspace $ \Vd$, $\Vc$, $\Wdd$, $\Wd$ is an absolutely
irreducible representation of $\Gamma$ acting on $\Q$.  
The kernels of these representations are ${\bf 1}, \ES_\Na \times {\bf 1}, {\bf 1}
\times\ES_\No, \Gamma$, respectively.  Since the kernels are unequal 
the representations are non-isomorphic.  A dimension count shows that 
\begin{equation} \label{Irrep_decomp}
\R^{\Na\No} = \Wd \oplus \Vc \oplus \Wdd \oplus \Vd. 
\end{equation}

Two related group-theoretic decom\-po\-si\-ti\-ons are also used later. 
See Theorems~\ref{T:consensus_bifur} and \ref{T:divided_deadlock_bifur}.

\begin{equation} \label{fix_decomp}
	\begin{array}{rcl}
		\Fix(\ES_\Na \times {\bf 1}) & = & \Vc\oplus\Wd \\ 
		\Fix( {\bf 1} \times \ES_\No ) & = & \Wdd\oplus\Wd,
	\end{array}
\end{equation}
where the {\em fixed-point subspace} $\Fix(H)$ of a subgroup $H \subseteq \Gamma$ 
is the set of all $z \in \Q$ such that $\alpha z = z$ for all $\alpha \in H$.

\subsection{Value-assignment interpretation of irreducible representations}

$ $
The subspaces in decomposition~\eqref{Irrep_decomp} admit the following interpretations:
\begin{itemize}
	\item[$\Wd$] is the {\it fully synchronous} subspace, where all agents assign the same value to every option. This is the subspace of {\em undecided} states that we will assume lose stability in synchrony-breaking bifurcations.
		
	\item[$\Vc$] is the {\it consensus} subspace, where all agents express the same preferences about the options. This subspace is the critical subspace associated to synchrony-breaking bifurcations leading to consensus value patterns described in \S\ref{case1}. 
		
	\item[$\Wdd$] is the  {\it deadlock} subspace, where each agent is deadlocked and agents are divided about the option values. This subspace is the critical subspace associated to synchrony-breaking bifurcations leading to deadlock value patterns described in \S\ref{case2}. 
	
	\item[$\Vd$] is the {\it dissensus} subspace, where agents express different preferences about the options. This subspace is the critical subspace associated to synchrony-breaking bifurcations leading to dissensus value patterns 
	described in \S\ref{case3a}, \S\ref{case3b}, \S\ref{case3c}. 
\end{itemize}

\subsection{Value assignment decompositions}
\label{SSEC: preliminary decom}

We also define two additional important value-assignment subspaces:
\begin{equation} \label{preference_deadlock}
\begin{array}{cclcl}
 \Vp & = & \Vc \oplus \Vd   \quad (\mbox{\em preference subspace})  \\
\Vdl & = & \Wd \oplus \Wdd  \quad (\mbox{\em indecision subspace}) 
\end{array}
\end{equation}
By~\eqref{Irrep_decomp} and \eqref{preference_deadlock}:
\[
\R^{\Na\No} = \Vp \oplus \Vdl
\]
and each $V_*$ is $\Gamma$-invariant.

Each introduced subspace and its 
decomposition admits a value-assignment interpretation.
By \eqref{distinct_irreps} and \eqref{preference_deadlock}, points in $\Vp$ have row sums equal to $0$ and generically have entries in a row that are not all equal for at least one agent. Hence, generically, for at least one agent some row elements are maximum values over the row and the corresponding options are favored by the given agent. That is, the subspace $\Vp$ consists of points where at least some agents express preferences among the various options.
In contrast, points in $\Vdl$ have entries in a row that are all equal, so all agents are deadlocked and the group in a state of indecision. This decomposition distinguishes rows and columns,  reflecting the asymmetry between agents and options from the point of view of value assignment.

\section{Parametrized value dynamics}
\label{S:PVD}

To understand how a network of indistinguishable agents values and forms preferences among a set of {\it a priori} equally valuable options, 
we model the transition between different value states of the influence network --- for example from fully synchronous  to consensus --- as a bifurcation. 
To do so, we introduce a {\it bifurcation parameter} $\lambda\in\R$ into
 model~\eqref{EQ: generic value evol},\eqref{EQ: admissible vf}, which leads to the parametrized dynamics
\begin{subequations}\label{EQ: general param decision dynamics}
	\begin{align}
		\dot\Zz&=\Gg(\Zz,\lambda)\\
		G_{ij}(\Zz,\lambda)&=G(z_{ij}, \overline{z_{\sf{A_{ij}}}}, \overline{z_{\sf{O_{ij}}}}, \overline{z_{\sf{E_{ij}}}},\lambda)\,.
	\end{align}	
\end{subequations}
We assume that the parametrized ODE~\eqref{EQ: general param decision dynamics} is also admissible and hence $\Gamma$-equivari\-ant. That is,
\[
\gamma\Gg(\Zz,\lambda)=\Gg(\gamma\Zz,\lambda)\,.
\]

Depending on context, the bifurcation parameter $\lambda$ can model a variety of environmental or endogenous parameters affecting the valuation process. In honeybee nest site selection, $\lambda$ is related to the rate of stop-signaling, a cross-inhibitory signal that enables value-sensitive decision making~\cite[Seeley {\it et al.}]{SKSHFM12}, \cite[Pais {\it et al.}]{PHSFLM13}. In animal groups on the move, $\lambda$ is related to the geometric separation between the navigational options~\cite[Biro {\it et al.}]{BSMG06}, \cite[Couzin {\it et al.}]{CIDGTHCLL11}, \cite[Leonard {\it et al.}]{LSNSCL12}, \cite[Nabet {\it et al.}]{NLCL09}. In phenotypic differentiation, $\lambda$ is related to the concentration of quorum-sensing autoinducer molecules~\cite[Miller and Bassler]{MB01}, \cite[Waters and Bassler]{WB05}. 
In speciation models, $\lambda$ can be an environmental parameter such as food availability, \cite[Cohen and Stewart]{CS00},
\cite[Stewart {\it et al.}]{SEC03}.
In robot swarms, $\lambda$ is related to the communication strength between the robots \cite[Franci {\it et al.}]{FBPL21}. In human groups, $\lambda$ is related to the attention paid to others~\cite[Bizyaeva {\it et al.}]{BFL22}.

\section{The undecided solution and its linearization}
\label{S:Deadlock_linearization}

It is well known that equivariant maps leave fixed-point subspaces invariant \cite[Golubitsky {\it et al.}]{GSS88}.  Since admissible 
maps $\Gg$ are $\Gamma$-equivariant, it follows that the $1$-dimensional undecided (fully synchronous) subspace 
$\Wd=\Fix(\Gamma)$ is flow-invariant; that is, invariant under the flow of {\em any} admissible ODE.   Let $v\in \Wd$ be nonzero.  Then for $x\in\R$ we can write 
\[
\Gg(xv,\lambda) = g(x,\lambda)v
\]
where $g:\R\times\R\to\R$. Undecided solutions are found by solving $g(x,\lambda) = 0$. Suppose that $g(x_0,\lambda_0) = 0$. We analyze synchrony-breaking bifurcations
from a path of solutions to $g(x,\lambda) = 0$ bifurcating at $(x_0,\lambda_0)$.

\begin{remark}
\label{r:admiss=equi}
In general, any linear admissible map is linear equivariant. In influence networks $\mathcal{N}_{mn}$,
all linear equivariant maps are admissible. More precisely,
 the space of linear admissible maps 
is $4$-dimensional (there are three arrow types and one node type). The space of
linear equivariants is also $4$-dimensional (there are four distinct absolutely irreducible representations \eqref{Irrep_decomp}). Therefore linear equivariant maps are the same
as linear admissible maps. This justifies applying equivariant methods in
the network context to find the generic critical eigenspaces.
\end{remark}

We now consider syn\-chro\-ny-breaking bifurcations.  
Since the four irreducible subspaces in \eqref{Irrep_decomp} are non-isomorphic
and absolutely irreducible, Remark~\ref{r:admiss=equi} implies that the 
Jacobian $J$ of \eqref{EQ: general param decision dynamics} at a fully synchronous equilibrium 
$(x_0v,\lambda_0)$ has up to four distinct real eigenvalues. Moreover, it is possible to find 
admissible ODEs so that any of these eigenvalues is $0$ while the others are negative.  
In short, there are four types of steady-state bifurcation from a fully synchronous equilibrium and each can be the first bifurcation.  When the critical eigenvector is in $\Wd$, the generic bifurcation is a saddle-node of fully synchronous states. The other three bifurcations correspond to the other three irreducible representations in \eqref{Irrep_decomp} and are synchrony-breaking.

\subsection{The four eigenvalues and their value-assignment interpretation}
\label{SS:3symm_breaking}

A simple calculation reveals that 
\begin{equation}
	\begin{array}{ccl}
		J|_{\Vd} & = & \cd I_{(\Na-1)(\No-1)}\\
		J|_{\Vc} & = & \cc  I_{\No-1}\\
		J|_{\Wdd} & = & \cdd  I_{\Na-1}\\
		J|_{\Wd} & = & \cdl \\
	\end{array}
\end{equation}
where
\begin{equation}\label{EQ: deadlocked jac eigvals}
	\begin{array}{ccl}
		\cd & = & \alpha-\beta-\gamma+\delta\\
		\cc & = & \alpha-\beta+(\Na-1)(\gamma-\delta)\\
		\cdd & = & \alpha-\gamma+(\No-1)(\beta-\delta)\\
		\cdl & = & \alpha+(\No-1)\beta+(\Na-1)\gamma+(\Na-1)(\No-1)\delta
	\end{array}
\end{equation}
and 
\[
\alpha=\frac{\partial G_{ij}}{\partial z_{ij}},\quad \beta=\frac{\partial G_{ij}}{\partial z_{il}},\quad
\gamma=\frac{\partial G_{ij}}{\partial z_{kj}},\quad \delta=\frac{\partial G_{ij}}{\partial z_{kl}}.
\]
Here $1\leq i,k \leq\Na$ (with $k\neq i$); $1 \leq  j,l \leq \No$ (with $l\neq j$); and all partial derivatives are evaluated at the fully synchronous equilibrium $(x_0v,\lambda_0)$.  The parameters
 $\beta,\gamma,\delta$ are the linearized weights of row, column, and diagonal arrows, respectively,
 and $\alpha$ is the linearized internal dynamics.

We can write \eqref{EQ: deadlocked jac eigvals} in matrix form as
\[
\Matrix{ \cd \\ \cc \\ \cdd \\ \cdl} = 
\Matrix{ 1 & -1 & -1 & 1 \\ 
1 & -1 & \Na - 1 & 1 - \Na \\ 
1 & \No - 1 &  - 1 & 1 - \No \\
1 & \No - 1 & \Na - 1 & (\Na - 1)( \No - 1)}
\Matrix{ \alpha \\ \beta \\ \gamma \\ \delta} \equiv L\Matrix{ \alpha \\ \beta \\ \gamma \\ \delta}
\]
Since $\det(L) = -\Na^2\No^2 \neq 0$ the matrix $L$ is invertible. Therefore
any of the three synchrony-breaking bifurcations can occur as the first bifurcation 
by making an appropriate choice of the partial derivatives $\alpha,\beta,\gamma,\delta$ of $\Gg$.

For example, synchrony-breaking dissensus bifurcation ($\cd = 0$) can occur from a stable deadlocked equilibrium ($\cc,\cdd,\cdl < 0$) if we choose $\alpha,\beta,\gamma,\delta$ by 
\[
\Matrix{ \alpha \\ \beta \\ \gamma \\ \delta} = L^{-1} \Matrixr{ 0 \\ -1 \\ -1 \\ -1} 
\]

\section{$\ES_n$ consensus and $\ES_m$ deadlock bifurcations}
\label{S:DLB}

We summarized synchrony-breaking consensus and deadlock bifurcations in Sections~\ref{case1} and~\ref{case2}.
We now revisit these bifurcations with more mathematical detail.
Both types of bifurcation reduce mathematically to equivariant bifurcation for the natural
permutation action of $\ES_m$ or $\ES_n$. This is
not immediately clear because of potential network constraints, but
in these cases it is easy to prove~\cite{AS07} that the
admissible maps (nonlinear as well as linear)
are the same as the equivariant maps. In Section~\ref{S:diss_bif} 
we see that the same comment does not apply to dissensus bifurcations.

The Equivariant Branching Lemma~\cite[XIII, Theorem~3.3]{GSS88}
states that generically, for every axial subgroup $\Sigma$, there exists a branch of steady-state 
solutions with symmetry $\Sigma$. The branching occurs from a trivial (group invariant) equilibrium.  
More precisely, let $\Gamma$ be a group acting on $\R^n$ and let 
\begin{equation} \label{bifurcation_family}
\dot{x} = f(x,\lambda)
\end{equation} 
where $f$ is $\Gamma$-equivariant, $f(x_0,\lambda_0) = 0$ (so $x_0$ is a trivial 
solution; that is, $\Gamma$ fixes $x_0$), and $x_0$ is a point of steady-state bifurcation.  
That is, if $J$ is the Jacobian of $f$ at $x_0$, then $K=\ker(J)$ is nonzero.  Generically, $K$ is an 
absolutely irreducible representation of $\Gamma$.  A subgroup $\Sigma\subset\Gamma$ 
is {\em axial relative to} a subspace $K\subset\Q$ if $\Sigma$ is an isotropy subgroup of $\Gamma$ acting 
on $K$ such that 
\[
\dim\Fix_K(\Sigma) = 1,
\] 
where the fixed-point subspace $\Fix_K(H)$ of a subgroup $H\subset\Gamma$ {\em relative to} $K$ is the set of all $z\in K$ such that $\alpha z=z$ for all $\alpha\in H$.
Suppose that a $\Gamma$-equivariant map has a $\Gamma$-invariant equilibrium at $x_0$ and that the kernel of the Jacobian at $x_0$ is an absolutely irreducible subspace $V$. Then generically, for each axial subgroup $\Sigma$ of $\Gamma$ acting on $V$, a branch of equilibria 
with symmetry $\Sigma$ bifurcates from $x_0$. Therefore all conjugate branches also occur,
as discussed in Section~\ref{S:sol_orbits}.

In principle there could be other branches of equilibria~\cite[Field and Richardson]{FR89} and other 
interesting dynamics~\cite[Guckenheimer and Holmes]{GH88}. For example,  \cite[Dias and Stewart]{DS03} consider
secondary bifurcations to $\ONE \times(\ES_p\times\ES_q\times\ES_{n-p-q})$, 
which are not axial.
We focus on the equilibria
given by the Equivariant Branching Lemma because these are known
to exist generically.

\subsection{Balanced colorings, quotient networks, and ODE-equivalence}
The ideas of `balanced' coloring and `quotient network' were introduced in \cite[Golubitsky {\it et al.}]{GST05}, \cite[Stewart {\it et al.}]{SGP03}.  
See also \cite[Golubitsky and Stewart]{GS06,GS22}.

Let $z$ be a network node and let $I(z)$ be the {\em input} set of $z$ consisting of all arrows whose head node is $z$.
Suppose that the nodes are colored. The coloring is {\em balanced} if any two nodes $z_1, z_2$ 
with the same color have {\em color isomorphic} input sets.  That is, there is a bijection $\sigma:I(z_1)\to I(z_2)$ 
such that the tail nodes of $a\in I(z_1)$ and $\sigma(a)\in I(z_2)$ have the same color.  

It is shown in the above references that when the coloring is balanced the subspace $Q$, where two node values are equal whenever the nodes
have the same color, is a flow-invariant subspace for every admissible vector field. The subspaces $Q$ are
the network analogs of fixed-point subspaces of equivariant vector fields.  Finally, it is also shown that $Q$  
is the phase space for a network whose nodes are the colors in the balanced coloring.  This
network is the {\em quotient network} (Figure~\ref{FIG: influence network_N23}, right). Through identification of same-color nodes in Figure~\ref{FIG: influence network_N23}, center, the quotient network exhibits self-coupling and two different arrows (dashed gray and dashed black) between pairs of nodes in such a way that the number of arrows between colors is preserved. For example, each red node of the original network receives a solid arrow from a red node, a dashed gray arrow and a dashed black arrow from a white node, and a dashed gray arrow and a dashed black arrow from a blue node, both in Figure~\ref{FIG: influence network_N23}, center, and in Figure~\ref{FIG: influence network_N23}, right. Similarly, for the other colors.

Two networks with the same number of nodes are {\em ODE-equivalent} if they have the same spaces of 
admissible vector fields.  \cite[Dias and Stewart]{DS05} show that two networks are ODE-equivalent if and 
only if they have the same spaces of linear admissible maps.  It is straightforward to check that the
networks in Figures~\ref{FIG: influence network_N23} (right) and \ref{FIG: influence network_N23_quot} 
have the same spaces of linear admissible maps, hence are ODE-equivalent. Therefore the bifurcation 
types from fully synchronous equilibria are identical in these networks.  These two ODE-equivalent 
networks, based on the influence network $\mathcal N_{23}$, can help to illustrate the bifurcation result in 
Theorem~\ref{T:consensus_bifur}.

\begin{figure}
	\centerline{
	\includegraphics[width=0.22\textwidth]{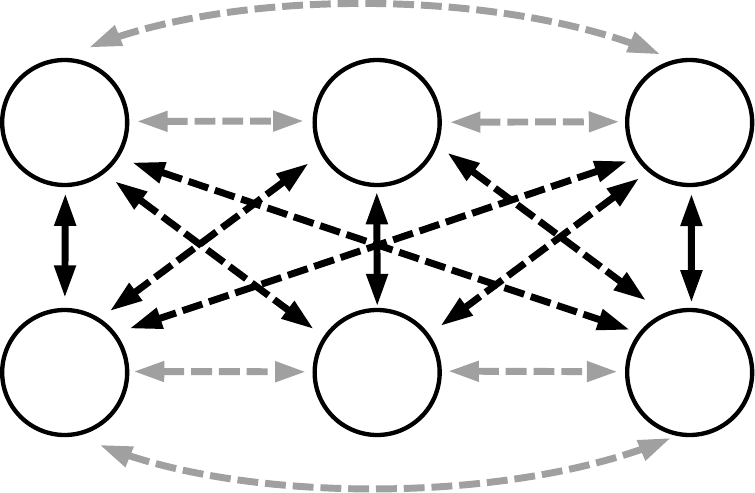}\qquad
	\includegraphics[width=0.22\textwidth]{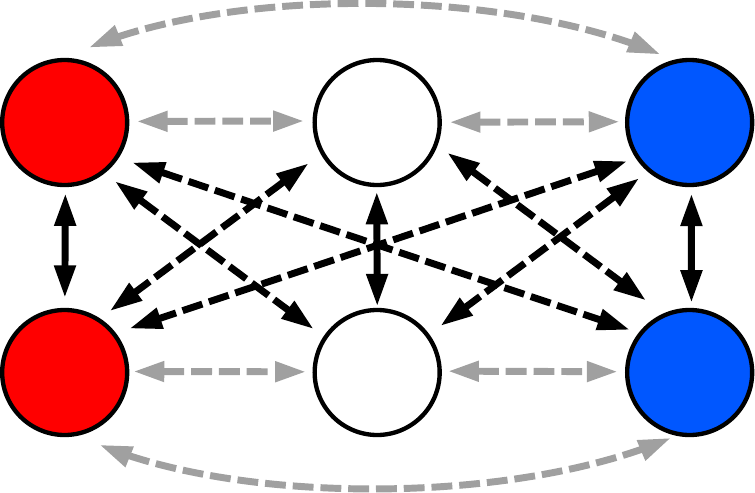}\qquad
	\includegraphics[width=0.22\textwidth]{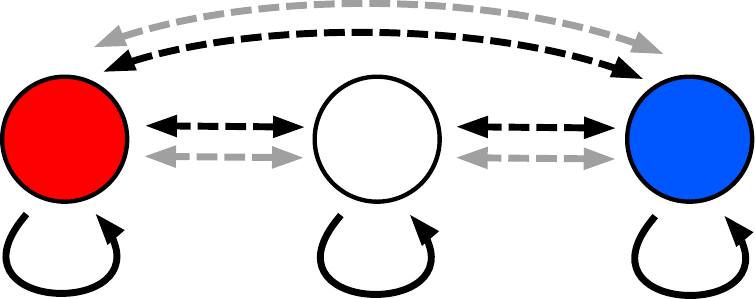}}
	\caption{(Left) The influence network $\mathcal{N}_{23}$ has three distinct arrow types: gray dashed row arrow, solid black column arrow, and black dashed diagonal arrow.
		(Middle) Balanced three coloring given by $\Fix(1\times \ES_2)$. (Right) Three-node quotient network $G_{\rm c}$ of (middle).
	 }
	\label{FIG: influence network_N23}
\end{figure}

\begin{figure}
	\centerline{
	\includegraphics[width=0.22\textwidth]{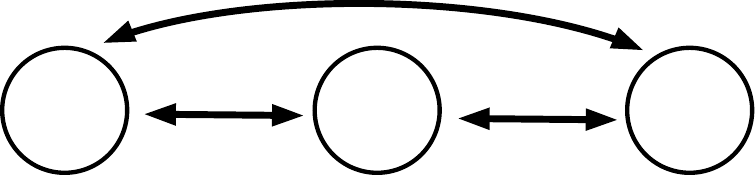}}
	\caption{Three-node quotient network with $\ES_3$-symmetry that is ODE-equivalent to $G_{\rm c}$ 
	illustrated in Figure~\ref{FIG: influence network_N23} (right).}
	\label{FIG: influence network_N23_quot}
\end{figure}

\subsection{Consensus bifurcation ($\cc = 0$)}
\label{SS:consensus}

Branches of equilibria stemming from consensus bifurcation can be proved to exist 
using the Equivariant Branching Lemma applied to a suitable quotient network.
The branches can be summarized as follows:

\begin{theorem} \label{T:consensus_bifur}
Generically, there is a branch of equilibria corresponding to the axial subgroup 
${\bf 1} \times (\ES_k\times \ES_{\No - k})$ for all $1 \leq k\leq \No-1$.  
These solution branches are tangent to $\Vc$, lie in the subspace 
$\Fix(\ES_{\Na}\times\ONE)=\Vc\oplus\Wd$, and are consensus solutions.  
\end{theorem}

\begin{proof} 
We ask: What are the steady-state branches that bifurcate from the fully synchronous state when $\cc=0$?  Using network theory we show in four steps that the answer reduces to $\ES_{\No}\cong {\bf 1} \times \ES_\No$ equivariant theory.  Figures~\ref{FIG: influence network_N23} and \ref{FIG: influence network_N23_quot} illustrate the method
when $(m,n) = (2,3)$.

\begin{enumerate} 
\item Let $G_{\rm c}$ be  the quotient network
 determined by the balanced coloring where all agents for a given option have the same color and different options are assigned different colors.
 The quotient $G_{\rm c}$ is a homogeneous all-to-all coupled $\No$-node network with three different arrow-types and multiarrows between some nodes.
 The $\Vc$ bifurcation occurs in this quotient network. 

\item The network $G_{\rm c}$ is ODE-equivalent to the standard all-to-all coupled $n$ simplex $\GG_n$ with no multi-arrows (Figure~\ref{FIG: influence network_N23_quot}). Hence the steady-state bifurcation theories for networks $G_{\rm c}$ and $\GG_n$ are identical.  

\item The admissible maps for the standard $n$-node simplex network $\GG_n$ are identical to the $\ES_n$-equivariant maps acting on $\R^{n-1}$.  Using the Equivariant Branching Lemma it is known that, generically, branches of equilibria bifurcate from the trivial (synchronous) equilibrium with isotropy subgroup $\ES_k\times \ES_{n - k}$ symmetry, where $1 \leq k \leq \No-1$;
see \cite[Cohen and Stewart]{CS00}.  Consequently, generically, consensus deadlock-breaking bifurcations lead to steady-state branches of solutions with $ {\bf 1}\times (\ES_k\times \ES_{\No - k})$ symmetry.  

\item The bifurcating branches of equilibria are tangent to the critical eigenspace $\Vc$.  Additionally \eqref{fix_decomp} implies that the admissible map leaves the subspace $\Fix(\ES_\Na \times {\bf 1}) = \Vc \oplus\Wd$ invariant.  Hence, the solution branches lie in the subspace $\Vc\oplus\Wd$ and consists of arrays with all rows identical. See \eqref{distinct_irreps}. 
\end{enumerate}
\end{proof}

\begin{remark} \rm
As a function of $\lambda$ each bifurcating consensus branch is tangent to $\Vc$. Hence, consensus bifurcation branches are ZRS and each agent values $k$ options higher and the remaining lower in such a way that the average value remains constant to linear order in $\lambda$ along the branch.
Since the bifurcating solution branch is in  $\Vc\oplus\Wd$, the bifurcating states have all rows equal; that is, all agents 
value the options in the same way. In particular, all agents favor the same $k$ options, disfavor the remaining $n-k$ ones, and the agents are in a consensus state.
Symmetry therefore predicts that in breaking indecision toward consensus, the influence network transitions from an undecided one option-cluster state to a two option-cluster consensus state made of favored and disfavored options. Intuitively, this happens because the symmetry of the fully synchronous undecided state is lost gradually, in the sense that the bifurcating solutions still have large (axial) symmetry group. Secondary bifurcations can subsequently lead to state with smaller and smaller symmetry in which fewer and fewer options are favored.
\end{remark}

\subsection{Deadlock bifurcation ($\cdd = 0$)}
\label{SS:deadlock}

The existence of branches of equilibria stemming from deadlock bifurcation can be summarized as follows:

\begin{theorem} \label{T:divided_deadlock_bifur}
Generically, there is a branch of equilibria corresponding to the axial subgroup 
$(\ES_k\times \ES_{\Na - k}) \times {\bf 1}$ for $1 \leq k\leq \Na-1$.  These solution branches are 
tangent to $\Wdd$, lie in the subspace  $\Fix(\ONE\times\ES_n)=\Wdd\oplus\Wd$, and are deadlock solutions.  
\end{theorem}

The proof is analogous to that of consensus bifurcation in Theorem~\ref{T:consensus_bifur}.

\begin{remark} \rm
As a function of $\lambda$ each bifurcating solution branch is tangent to $\Wdd$. Hence, the column sums vanish to first order in $\lambda$. It follows that generically at first order in $\lambda$ each column has some unequal entries 
and therefore agents assign two different values to any given option.
This means that agents are divided about option values.  Since the bifurcating solution branch is in $\Wdd\oplus\Wd$, the bifurcating states have all columns equal; that is,
each agent assigns the same value to all the options, so each agent is deadlocked.  In particular, $k$ agents favor 
all options and $m-k$ agents disfavor all options.
Symmetry therefore predicts that in breaking indecision toward deadlock, the influence network transitions from an undecided one agent-cluster state with full symmetry to a two agent-cluster deadlock state with axial symmetry made of agents that favor all options and agents that disfavor all options. Secondary bifurcations can subsequently lead to state with smaller and smaller symmetry in which fewer and fewer agents favor all options.
\end{remark}

\subsection[Stability of consensus and deadlock bifurcation branches]{Stability of consensus  and deadlock bifurcation branches}
\label{S:SCB}

In consensus bifurcation, all rows of the influence matrix $(z_{ij})$ are identical and have sum zero. 
In deadlock bifurcation, all columns of the influence matrix $(z_{ij})$ are identical and have sum zero.  
As shown in Sections~\ref{SS:consensus} and \ref{SS:deadlock} these problems abstractly reduce to 
$\ES_N$-equivariant bifurcation on the nontrivial irreducible subspace
\[
\{x \in \R^N : x_1 + \cdots + x_N = 0 \} 
\]
where $N = m$ for deadlock bifurcation and $N = n$ for consensus bifurcation.

This bifurcation problem has been studied extensively as a model for sympatric speciation in
evolutionary biology~\cite[Cohen and Stewart]{CS00}, \cite[Stewart {\it et al.}]{SEC03}. Indeed, 
this model can be viewed as a decision
problem in which the agents are coarse-grained tokens for organisms, initially of the same 
species (such as birds), which assign preference values to some phenotypic variable (such as beak length)
in response to an environmental parameter $\lambda$ such as food availability.
It is therefore a special case of the models considered here, with $m$ agents and one option.
The primary branches of bifurcating states correspond to the axial subgroups,
which are (conjugates of) $\ES_p \times \ES_q$ where $p+q = N$.

Ihrig's Theorem \cite[Theorem 4.2]{IG84} or \cite[Chapter XIII Theorem 4.4]{GSS88}
shows that in many cases transcritical branches of solutions that are obtained 
using the Equivariant Branching Lemma are unstable at bifurcation.  See Figure~\ref{fig:Ihrig}.
Indeed, this is the case for axial branches for the $\ES_N$ bifurcations, $N>2$, and at first sight this theorem 
would seem to affect their relevance. However, simulations show that equilibria with these
synchrony patterns can exist and be stable. They arise by jump bifurcation, and the axial branch 
to which they jump has regained stability by a combination of two methods: 

(a) The branch `turns round' at a saddle-node.

(b) The stability changes when the axial branch meets a secondary branch. Secondary branches 
correspond to isotropy subgroups of the form $\ES_a \times \ES_b \times \ES_c$ where $a+b+c = N$.

\begin{figure}
	\centering
	\includegraphics[width=0.6\textwidth]{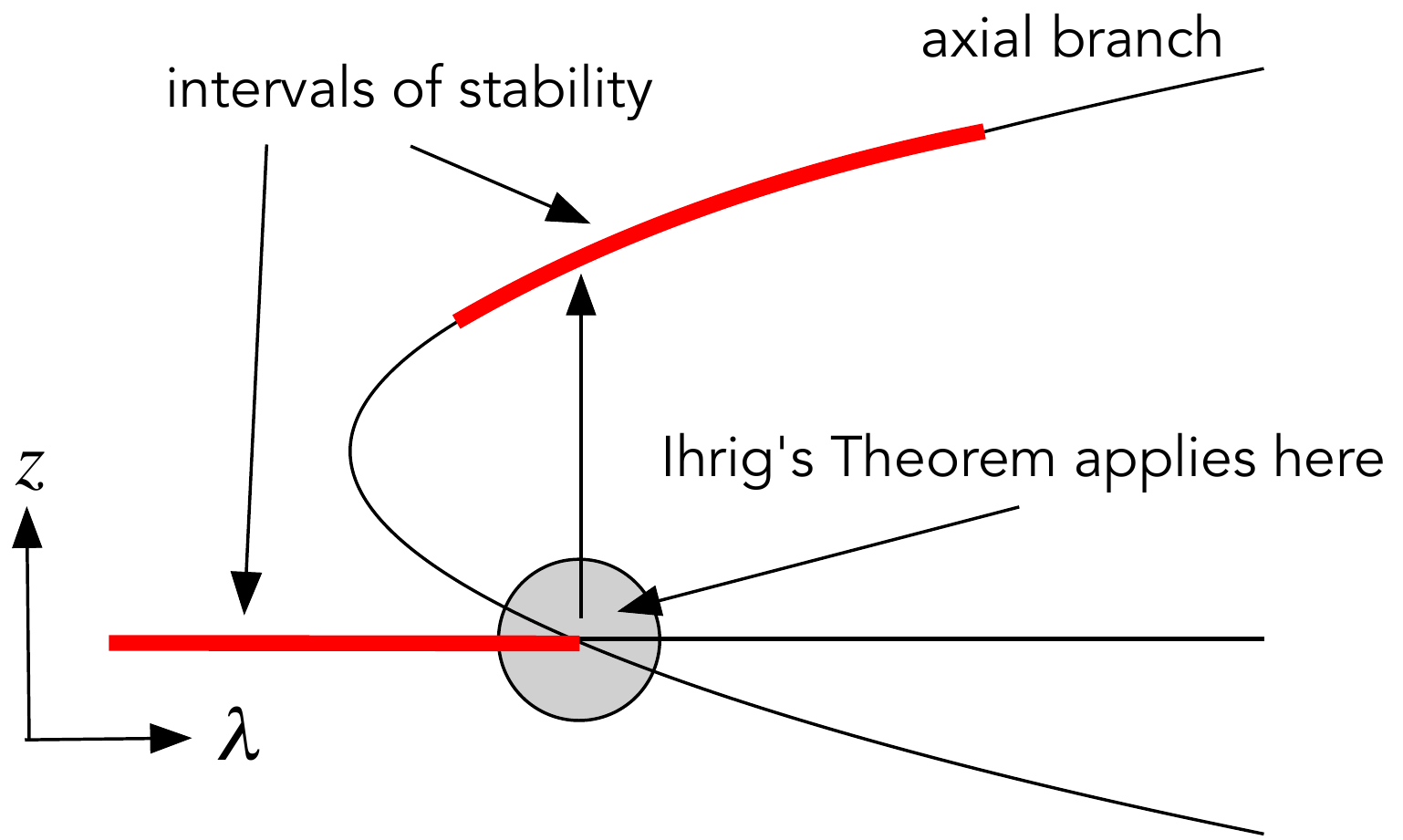}
	\caption{Sketch illustrating solution stability near, but not at, bifurcation. Such bifurcations lead 
	to jump transitions rather than smooth transitions. Vertical $z$ coordinate is multidimensional.}
	\label{fig:Ihrig}
\end{figure}

The fact that these axial solutions can both exist and be stable in model equations is shown 
in~\cite[Elmhirst]{E02} and \cite[Stewart et al.]{SEC03}.  Simulations of consensus and 
deadlock bifurcations are discussed in Section~\ref{S:simulation} and 
simulations of dissensus bifurcations are discussed in Section~\ref{S:simulation_dissensus}.
The main prediction is that despite Ihrig's Theorem, axial states can and do occur stably.
However, they do so through a jump bifurcation to an axial branch that has regained stability, 
not by local movement along an axial branch.  This stability issue appears again in dissensus 
bifurcation and simulations in Section~\ref{SS:stability_dissensus} show that axial solutions 
can be stable even though they are unstable at bifurcation.

\section{Axial Balanced Colorings for Homogeneous Networks}
\label{S:axial}

The analysis of branches of bifurcating solutions in the dissensus subspace $\Vd$ 
requires a natural network analog of the Equivariant Branching Lemma. This new version applies to exotic colorings (not determined 
by a subgroup $\Sigma$) as well as orbit colorings (determined by a subgroup $\Sigma$).
See Section~\ref{S:SBSB}. We deal with the generalization here, and apply it to $\Vd$ in Section~\ref{S:diss_bif}.

For simplicity, assume that each node space of $\GG$ is $1$-dimensional, and 
that $f$ in \eqref{bifurcation_family} is a $1$-parameter family of admissible 
maps. Let $\Delta$ be the diagonal space of fully synchronous states.  By 
admissibility, $f:\Delta\times\R\to\Delta$. 
Hence we can assume generically 
that $f$ has a trivial steady-state bifurcation point at $x_0\in\Delta$.
Given a coloring $\bowtie$, its {\em synchrony subspace} $\Delta_{\bowtie}$ is the subspace where all nodes with the same color are synchronous.

\begin{definition} \em \label{D:axial_netwk_reg}
Let $K$ be the kernel (critical eigenspace for steady-state
bifurcation) of the Jacobian $J = \DD f|_{x_0,\lambda_0}$. 
Then a balanced coloring\index{balanced!coloring} $\bowtie$ with 
synchrony subspace $\Delta_{\bowtie}$ is {\em axial relative to $K$} if 
\begin{equation}
\label{D:axial_coloring}
\begin{array}{rcl}
K \cap \Delta & = & \{0\} \\
\dim(K \cap \Delta_{\bowtie}) & = & 1
\end{array}
\end{equation}
\end{definition}

\subsection{The synchrony-breaking branching lemma}

We now state and prove the key bifurcation theorem for this paper.
The proof uses the method of Liapunov--Schmidt reduction and various
standard properties \cite[Chapter VII Section 3]{GSS88}.

\begin{theorem} \label{T:simp_eigen_reg}
With the above assumptions and notation, let $\bowtie$ be an axial 
balanced coloring. Then generically a unique branch of equilibria with synchrony 
pattern\index{synchrony!pattern} $\bowtie$ bifurcates from $x_0$ at $\lambda_0$. 
\end{theorem}

\begin{proof}
The first condition in \eqref{D:axial_coloring} implies that the restriction $f:\Delta\times\R\to\Delta$ is 
nonsingular at $(x_0,\lambda_0)$. Therefore by the Implicit Function Theorem there is a 
branch of trivial equilibria $X(\lambda)$ in $\Delta$ for $\lambda$ near $\lambda_0$,  
so $f(X(\lambda),\lambda) \equiv 0$ where $x_0 = X(\lambda_0)$. We can translate the bifurcation 
point to the origin so that without loss of generality we may assume $f(0,\lambda) = 0$ for
all $\lambda$ near $0$.

The second condition in \eqref{D:axial_coloring} implies that $0$ is a simple eigenvalue of 
$J' = J|_{\Delta_{\bowtie}}$.  Therefore Liapunov--Schmidt reduction of the restriction
$f:\Delta_{\bowtie}\times\R\to\Delta_{\bowtie}$ leads to a reduced map
$g:\R\{v\}\times\R \to \R\{v^*\}$
where $\Delta_{\bowtie} \cap K = \R(v)$ and $v^*$ is the left eigenvector of $J'$ for eigenvalue $0$. The zeros 
of $g$ near the origin are in 1:1 correspondence with the zeros of $f|_{\Delta_{\bowtie}\times\R}$ 
near $x_0 = 0$.  We can write $g(sv,\lambda) = h(s,\lambda)v^*$.
By~\cite[Stewart and Golubitsky]{SG11},  Liapunov--Schmidt reduction can be chosen to preserve the existence of a trivial solution, 
so we can assume $h(0,\lambda) = 0$.

General properties of Liapunov--Schmidt reduction \cite[Golubitsky and Stewart]{GS85} and the fact that $x\mapsto \lambda x$
is admissible for each $\lambda$, imply that $h_\lambda(0,0)$ is generically nonzero. The Implicit Function 
Theorem\index{Implicit Function Theorem} now implies the existence of a unique branch 
of solutions $h(\Lambda(\lambda),\lambda) \equiv 0$ in $\Delta_{\bowtie}$; that is, with 
$\bowtie$ synchrony. Since $K \cap \Delta = 0$ this branch is not a continuation of the trivial branch.

The uniqueness statement in the theorem shows that the synchrony pattern
on the branch concerned is precisely $\bowtie$ and not some coarser pattern.
\end{proof}

Let the network symmetry group be $\Gamma$ and let $\Sigma$ be a subgroup of $\Gamma$. 
Since $K$ is $\Gamma$-invariant,  
$\Fix_K(\Sigma) = \Fix_{\Q}(\Sigma)\cap K$.  
Therefore, if $\Fix_{\Q}(\Sigma)\cap K$ is $1$-dimensional then $\Sigma$ is axial on $K$ and 
Theorem~\ref{T:simp_eigen_reg} reduces to the Equivariant Branching Lemma.

\begin{remark}
	\label{r:subtlety}
	Figure~\ref{F:Qpattern} shows two balanced patterns on a $2 \times 4$ array.
	On the whole space $\R^2\otimes\R^4$ these are distinct, corresponding to
	different synchrony subspaces $\Delta_1, \Delta_2$, which are flow-invariant.
	When intersected with $\Vd$, both patterns give the same
	1-dimensional space, spanned by
	\[
	\Matrixr{1 & -1 & 0 & 0 \\ -1 & 1 & 0 & 0}
	\]
	By Theorem~\ref{T:simp_eigen_reg} applied to $\Vd$,
	there is generically a branch
	that is {\em tangent} to the kernel $\Vd$, and lies in $\Delta_1$,
	and a branch that lies in $\Delta_2$.
	However, $\Delta_1 \subseteq \Delta_2$.
	Since the bifurcating branch is locally unique, so 
	it must lie in the intersection of those spaces; that is, it corresponds
	to the synchrony pattern with fewest colors, Figure~\ref{F:Qpattern} (left).
		\begin{figure}[htb]
		\centerline{
			\includegraphics[width=.4\textwidth]{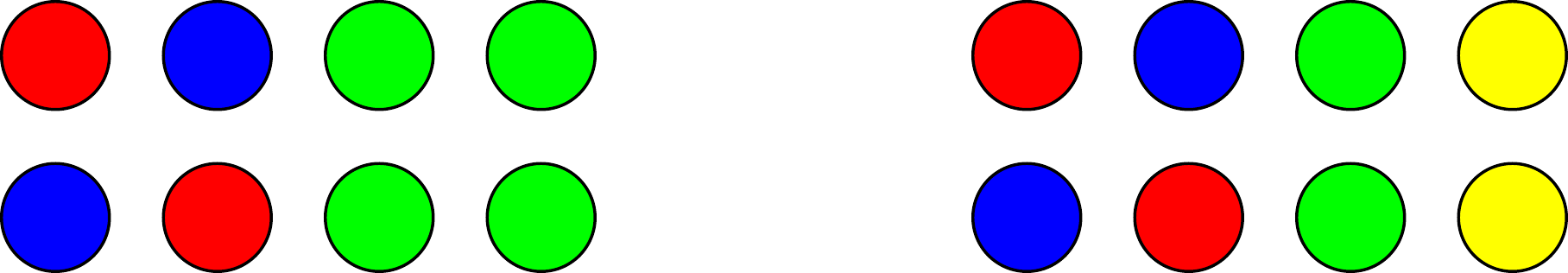}}
		\caption{Two axial patterns (for $\Vd$) with the same linearization on $\Vd$.}
		\label{F:Qpattern}
	\end{figure}
\end{remark}

 \section{Dissensus bifurcations ($\cd = 0$)}
 \label{S:diss_bif}
 
 Our results for axial dissensus bifurcations, with critical eigenspace 
 $\Vd$, have been summarized and interpreted in~\S\ref{case3}
 We now state the results in more mathematical language and  use Theorem~\ref{T:simp_eigen_reg} to determine solutions 
 that are generically spawned at a 
 deadlocked dissensus bifurcation: see Theorem~\ref{T:diss}. In this section we
 verify that Cases \S\ref{case3a}, \S\ref{case3b}, \S\ref{case3c} are all the possible dissensus axials. 
 
The analysis leads to a combinatorial structure that arises naturally
in this problem. It is a generalization of the concept of a Latin square,
and it is a consequence of the balance conditions for colorings.
 
 \begin{definition}\em
\label{D:LatinRec}
A {\em Latin rectangle} is an $a \times b$ array of colored nodes,
such that:

(a) Each color appears the same number of times in every row.

(b) Each color appears the same number of times in every column.
\end{definition}

Condition (a) is equivalent to the row couplings being balanced.
Condition (b) is equivalent to the column couplings being balanced.

Definition~\ref{D:LatinRec} is not the usual definition of `Latin rectangle', which does not
permit multiple entries in a row or column and imposes other conditions~\cite[wikipedia]{wikib}.

Henceforth we abbreviate colors by $R$ (red), $B$ (blue), $G$ (green), and $Y$ (yellow).
The conditions in Definition~\ref{D:LatinRec} are independent. 
Figure~\ref{F:LatinRec} (left) shows a $3 \times 6$ Latin rectangle with $(R,B,G)$ columns and 
$(R,R, B,B,G,G)$ rows.  Counting colors shows that 
Figure~\ref{F:LatinRec} (right) satisfies (b) but not (a).  

In terms of balance: in Figure~\ref{F:LatinRec} (left), each $R$ node
has one $R$, two $B$, and two $G$ row-arrow inputs; one $B$ and one
$G$ column-arrow input; and four $R$, three $B$, and three $G$
diagonal-arrow inputs. Similar remarks apply to the other colors.

In contrast, in Figure~\ref{F:LatinRec} (right), the $R$ node in the first
row has one $B$ and two $G$ row-arrow inputs, whereas the
$R$ node in the second
row has two $B$ row-arrow inputs and one $G$ row-arrow input.
Therefore the middle pattern is not balanced for row-arrows,
which in particular implies that it is not balanced.

\begin{figure}[htb]
\centerline{
\includegraphics[width = .25\textwidth]{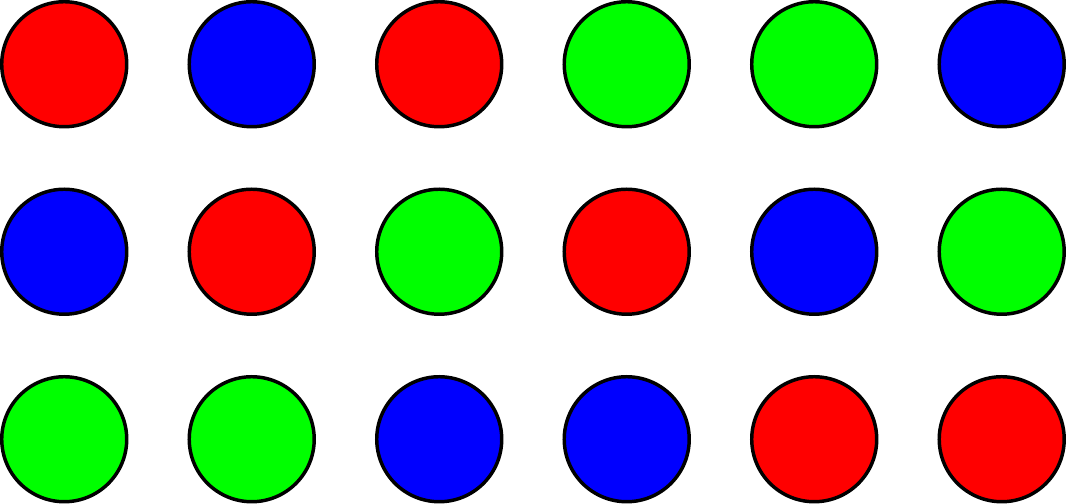} \qquad\qquad
\includegraphics[width = .19\textwidth]{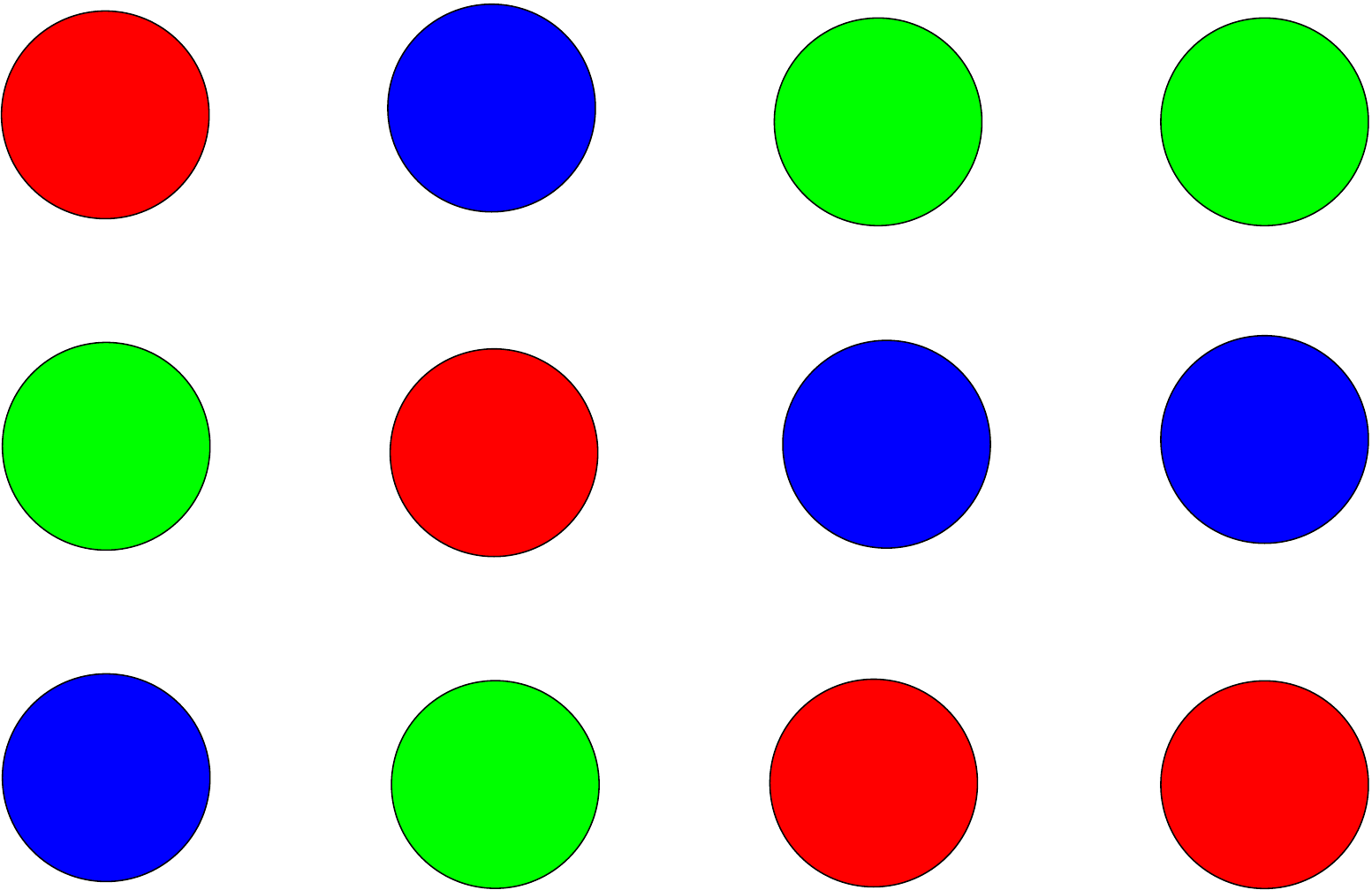}}
\caption{(Left) $3\times 6$ Latin rectangle with 3 colors. (Right) $3\times 4$  rectangle that satisfies (b) but not (a).} 
\label{F:LatinRec}
\end{figure}

The classification of axial colorings on influence networks is:

\begin{theorem}
\label{T:diss}
The axial colorings relative to the dissensus space $\Vd$, up to reordering rows and columns, are
\begin{itemize}
\item[\rm (a)]
$\bowtie\ =\Matrix{B_0 & B_1}$ where $B_0$ is a rectangle with one color ($Y$)
and $B_1$ is a Latin rectangle with two colors ($R,B$). The fraction $0<\rho<1$ of red nodes in every row of $B_1$ is the same as the fraction of red nodes in every column of $B_1$. Similarly for blue nodes. If in $\Delta_{\bowtie}\cap\Vd$ the value of yellow nodes is $z_Y$, the value of red nodes is $z_R$, and the value of blue nodes is $z_B$, then $z_Y=0$ and
\begin{equation}
	\label{E: case a zero sum}
	z_B=-\frac{\rho}{1-\rho}z_R.
\end{equation}
Possibly $B_0$ is empty.
\item[\rm (b)]
$\bowtie\ =\Matrix{B_{0}  \\ B_1}$ where $B_0$ is a rectangle with one color ($Y$)
and $B_1$ is a Latin rectangle with two colors  ($R,B$). The fraction $0<\rho<1$ of red nodes in every row of $B_1$ is the same as the fraction of red nodes in every column of $B_1$. Similarly for blue nodes. If in $\Delta_{\bowtie}\cap\Vd$ the value of yellow nodes is $z_Y$, the value of red nodes is $z_R$, and the value of blue nodes is $z_B$, then $z_Y=0$ and
	\begin{equation}
		\label{E: case b zero sum}
		z_B=-\frac{\rho}{1-\rho}z_R.
\end{equation}
Possibly $B_0$ is empty; if so, this pattern is the same as {\rm (a)} with empty $B_0$.

\item[\rm (c)]
$\bowtie\ =\Matrix{B_{11} & B_{12} \\ B_{21} & B_{22}}$ where the $B_{ij}$ are non-empty rectangles 
with only one color. Let $z_{ij}$ be the value associated to the color of $B_{ij}$ in $\Delta_{\bowtie}\cap\Vd$ and let $B_{11}$ be $r\times s$, $B_{12}$ be $r\times (n-s)$, $B_{21}$ be $(m-r)\times s$ and $B_{22}$ be $(m-r)\times (n-s)$. Then
\begin{subequations}
	\label{E:case c zero sum}
	\begin{align}
		z_{12}&=-\frac{s}{n-s} z_{11}\\
		z_{21}&=-\frac{r}{m-r} z_{11}\\
		z_{22}&=\frac{s}{n-s}\frac{r}{m-r}z_{11}
	\end{align}
\end{subequations}

\end{itemize}
\end{theorem}

\begin{remark}
The theorem implies that axial patterns of cases (a) and (b) either have two colors, one
corresponding to a negative value of $z_{ij}$ and the other to a positive value;
or, when the block $B_0$ occurs, there can also be a third value that is
zero to leading order in $\lambda$. Axial patterns of case (c) have four colors, two corresponding to positive values and two corresponding to negative values, when $r\neq\frac{m}{2}$ or $s\neq\frac{n}{2}$, or two colors, one corresponding to a positive value and the other to a negative value, when $r=\frac{m}{2}$ and $s=\frac{n}{2}$.
\end{remark}

\subsection{Proof of Theorem~\ref{T:diss}}
\label{S:PSTdiss}

We need the following theorem, whose proof follows directly from imposing the balanced coloring condition and can be found in~\cite{GS22}.

\begin{theorem}
	\label{T:main_rect}
	A coloring of $\mathcal{N}_{mn}$ is balanced if and only if it is conjugate under
	$\ES_m \times \ES_n$ to a tiling by rectangles, meeting along edges, such that
	
	{\rm (a)} Each rectangle is a Latin rectangle.
	
	{\rm (b)} Distinct rectangles have disjoint sets of colors.
\end{theorem}
See Figure~\ref{F:RectDecomp} for an illustration.

\begin{figure}[htb]
	\centerline{
		\includegraphics[width = .45\textwidth]{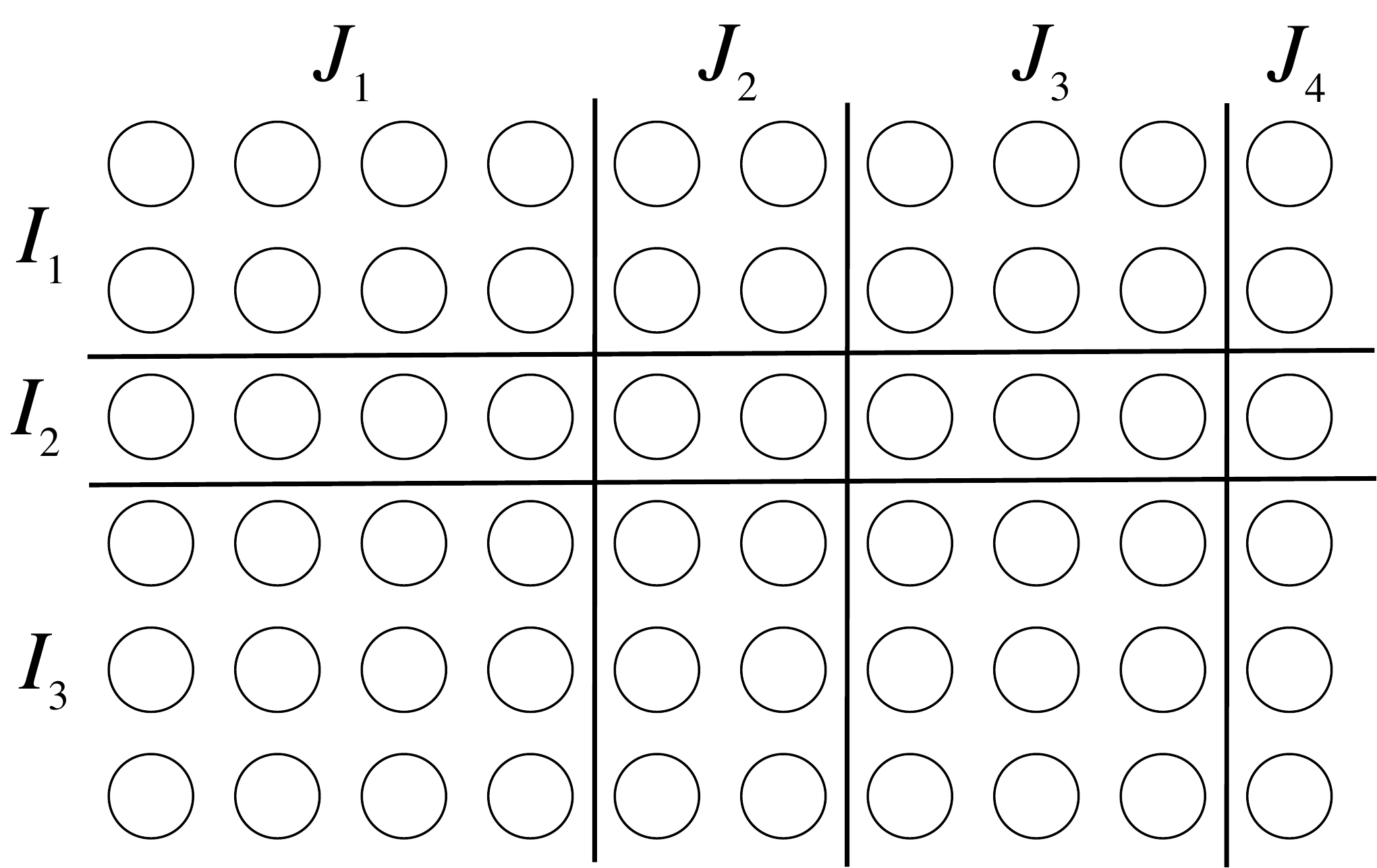}}
	\caption{Rectangular decomposition. Each sub-array must be a Latin rectangle,
		and distinct sub-arrays have disjoint color-sets.}
	\label{F:RectDecomp}
\end{figure}

\begin{proof}[Proof of Theorem~\ref{T:diss}] $ $
	Using Theorem~\ref{T:main_rect}, decompose the coloring  array $\bowtie$ into
	its component Latin rectangles $B_{ij}$. Here $1 \leq i \leq s$ and $1 \leq j \leq t$.
	Suppose that $B_{ij}$ contains $n_{ij}$ distinct colors.
	Let $z \in \Delta_{\bowtie} \cap \Vd$. This happens if and only if
	$z \in \Delta_{\bowtie}$ and all row- and column-sums of $z$ are zero.
	
	We count equations and unknowns to find the dimension of $\Delta_{\bowtie} \cap \Vd$. There
	are $s$ distinct row-sum equations and $t$ distinct column-sum equations.
	These are independent except that the sum of entries over all rows is the same as the sum
	over all columns, giving one linear relation. Therefore the number of independent
	equations for $z$ is $s+t-1$. The number of independent variables is $\sum_{ij} n_{ij}$.
	Therefore
	\[
	\dim (\Delta_{\bowtie} \cap \Vd) = \sum_{ij} n_{ij} - s - t + 1
	\]
	and coloring $\bowtie$ is axial if and only if
	\begin{equation}
		\label{E:axial_condition}
		\sum_{ij} n_{ij} = s + t.
	\end{equation}
	Since $n_{ij} \geq 1$ we must have $st \leq s+t$. It is easy to prove that
	this condition holds if and only if
	\[
	(s,t) = (1,t) \qquad (s,t) = (s,1) \qquad (s,t) = (2,2).
	\]
	
	We now show that these correspond respectively to cases (a), (b), (c)
	of the theorem.
	
	If $(s,t) = (2,2)$ then \eqref{E:axial_condition} implies that each $n_{ij} = 1$.
	This is case (c).
	
	If $(s,t) = (1,t)$ then $s+t-st = t+1-t = 1$. Now \eqref{E:axial_condition}
	requires exactly one $n_{1j}$ to equal 2 and the rest to equal 1.
	Now Remark~\ref{r:subtlety} comes into play. The
	columns with $n_{1j}= 1$ satisfy the column-sum condition, hence
	those $z_{1j} = 0$, and we can amalgamate those columns into a single
	zero block $B_0$. This is case (a).
	
	Case (b) is dual to case (a) and a similar proof applies.
	
	Finally, the ZRS and ZCS conditions give equations~\eqref{E: case a zero sum},
	\eqref{E: case b zero sum}, and \eqref{E:case c zero sum}.
\end{proof}

 \subsection{Orbit axial versus exotic axial}
 
 Axial matrices in Theorem~\ref{T:diss}(c) are orbital.  Let $B_{ij}$ be a $p_i\times q_j$ 
 matrix.  Let $\Sigma$ be the group generated
 by elements that permute rows $\{1,\ldots, p_1\}$, permute rows $\{p_1+1,\ldots,p_1+p_2\}$, permute columns $\{1,\ldots, q_1\}$, 
 and permute columns $\{q_1+1,\ldots, q_1+q_2\}$. Then the axial matrices in Theorem~\ref{T:diss}(c) are those in $\Fix(\Sigma)$.
 
 However, individual rectangles being orbital need
 not imply that the entire pattern is orbital.
  
 Axial matrices in Theorem~\ref{T:diss}(a) are orbital if and only if the $p\times q$ Latin rectangle $B_1$ is orbital.  Suppose $B_1$ is orbital, being fixed by the subgroup $T$ of $\ES_p\times\ES_q$.  Then $[B_0 \; B_1]$ is the fixed-point space of the subgroup $\Sigma$    
generated by $T$ and permutations $P$ of the columns in $B_0$. Hence $[B_0 \; B_1]$ is orbital.  Conversely, 
if $[B_0 \; B_1]$ is orbital, then the group $\tilde{T}$ fixes $[B_0 \; B_1]$.  There is a subgroup $T$ of $\tilde{T}$ that fixes the columns of $B_0$ and $\Fix(T)$ consists of multiples of $B_1$.  Thus $B_1$ is orbital.

\subsection{Two or three agents}
\label{S:2or3a}

We specialize to $2 \times n$ and $3 \times n$ arrays. Here it turns out that
all axial patterns are orbital. Exotic patterns appear when $m \geq 4$.
\subsubsection{$2 \times n$ array}

The classification can be read off directly from Theorem~\ref{T:diss},
observing that every 2-color $2 \times k$ Latin rectangle must
be conjugate to one of the form
\begin{equation}
\label{E:2xk_pattern}
\Matrix{R & R & \cdots & R & B & B & \cdots & B \\
	B & B & \cdots & B & R & R & \cdots & R}
\end{equation}
with $k$ even and $k/2$ nodes of each color in each row.

This gives the form of $B_1$ in Theorem~\ref{T:diss} (a). 
Here $\rho = 1/3$ and $q$ must be divisible by $3$. 
There can also be a zero block $B_0$
except when $n$ is even and $k = n/2$.
Theorem~\ref{T:diss} (b) does not occur.

For Theorem~\ref{T:diss} (c), the $B_{1j}$ must be $1 \times k$, and the $B_{2j}$ must be $1 \times (n-k)$.

Both types are easily seen to be orbit colorings. 

\subsubsection{$3 \times n$ array}
The classification can be read off directly from Theorem~\ref{T:diss},
observing that every 2-color $3 \times k$ Latin rectangle must (permuting colors if necessary)
be conjugate to one of the form
\begin{equation}
\label{E:3xk_pattern}
\Matrix{R & R & \cdots & R & B & B & \cdots & B & B & B & \cdots & B \\
	B & B & \cdots & B & R & R & \cdots & R & B & B & \cdots & B \\
	B & B & \cdots & B & B & B & \cdots & B & R & R & \cdots & R 
	}
\end{equation}
with $k = 3l$ and $k/3$ $R$ nodes and $2k/3$ $B$ nodes in each row.
Here $\rho = 1/3$ and $q$ must be divisible by $3$. 

Theorem~\ref{T:diss} (a): Border this with a zero block $B_0$ when $n \neq 3l$.

Theorem~\ref{T:diss} (b): $B_0$ must be $1 \times n$. Then $B_1$ must be
a $2 \times n$ Latin rectangle with 2 colors, already classified.
This requires $n = 2k$ with $k$ $R$ nodes and $k$ $B$ nodes in each row.
(No zero block next to this can occur.) 

Theorem~\ref{T:diss} (c):
Without loss of generality, $B_{11}$  is $2 \times k$ and $B_{21}$ is $2 \times (n-k)$,
while $B_{12}$  is $1 \times k$ and $B_{22}$ is $1 \times (n-k)$.
All three types are easily seen to be orbit colorings. 

\subsection{Two or three options}

When the number of options is $2$ or $3$ the axial patterns
are the transposes of those discussed in \S\ref{S:2or3a}.
The interpretations of these patterns are different, because the roles
of agents and options are interchanged.

 \subsection{Exotic patterns}
\label{S:EP}

For large arrays it is difficult to determine which $2$-color Latin rectangles
are orbital and which are exotic, because
the combinatorial possibilities for Latin rectangles explode and
the possible subgroups of $\ES_\Na \times \ES_\No$ also grow rapidly.
Here we show that exotic patterns exist. This has implications for
bifurcation analysis: using only the Equivariant Branching Lemma omits
all exotic branches. These are just as important as the orbital ones.

Figure~\ref{F:4x6_exotic} shows an exotic $4 \times 6$ pattern $\bowtie$.  To prove 
that this pattern is exotic we find its isotropy subgroup $H$ in $\ES_4 \times \ES_6$ 
and show that $\Fix(H)$ is not $\bowtie$. Theorem~\ref{T:exotic_suff} below shows 
that there are many exotic $4 \times n$ patterns for larger $n$, and provides
an alternative proof that this pattern is exotic.  

Regarding stability, we note that Ihrig's Theorem~\cite{IG84} applies only to orbital 
patterns, and it is not clear whether a network analog is valid.  Currently, we cannot 
prove that an exotic pattern can be stable near bifurcation though in principle such 
solutions can be stable in phase space.  See Section~\ref{SS:stability_balanced}.  
However, simulations show that both orbital axial branches and exotic axial branches 
can regain stability. See the simulations in Section~\ref{SS:stability_dissensus}.

\begin{figure}[htb]
\centerline{
\includegraphics[width=.25\textwidth]{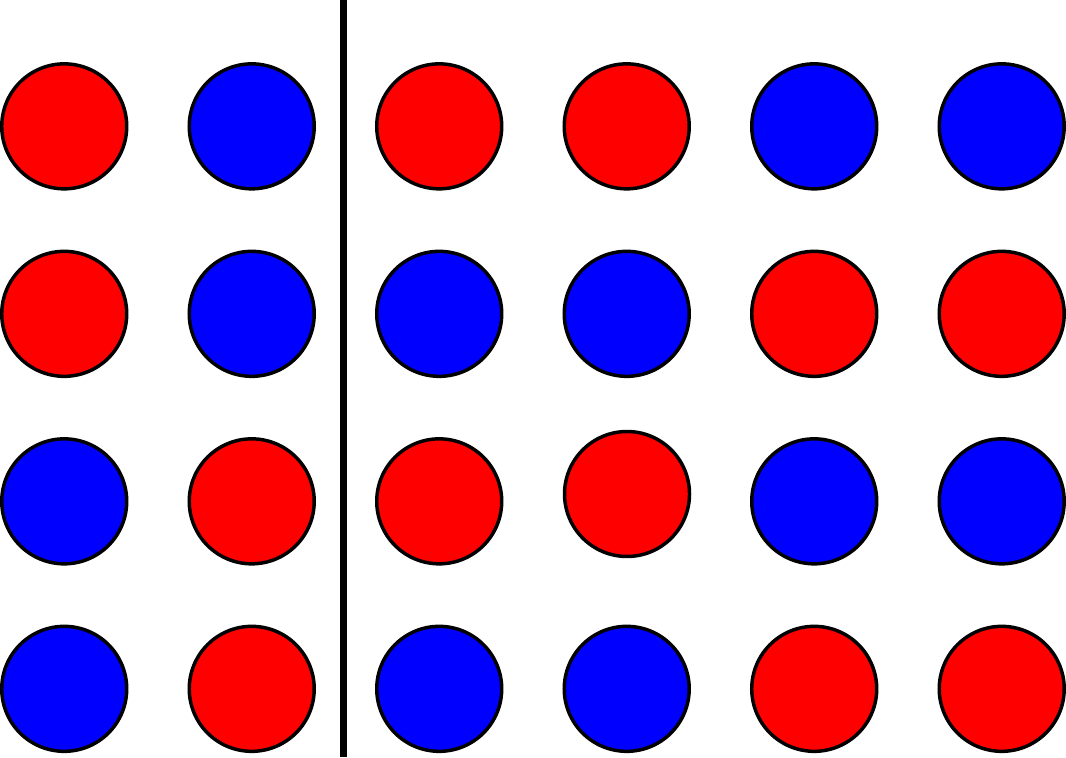} \qquad\qquad
\includegraphics[width=.25\textwidth]{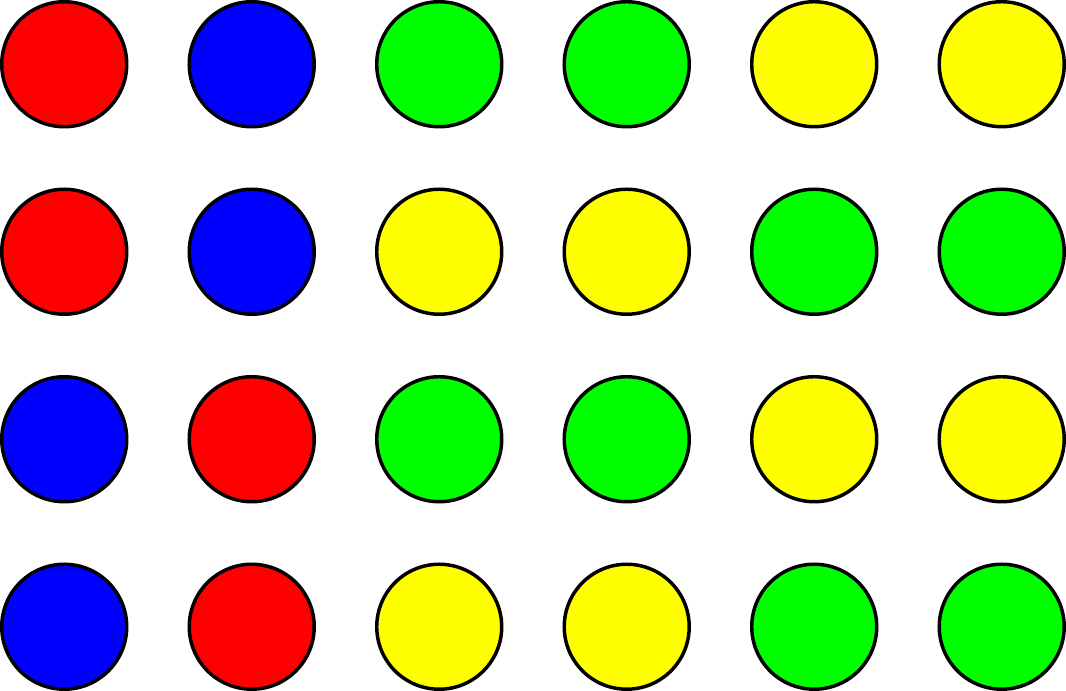}
}
\caption{(Left) Exotic $4 \times 6$ pattern. (Right) Pattern for $\Fix(H)$.}
\label{F:4x6_exotic}
\end{figure}

The general element of $\ES_4 \times \ES_6$ permutes both rows and columns.   
We claim: 
\begin{lemma}
Let $\sigma_{ij}$ interchange columns $i$ and $j$ in Figure~{\rm \ref{F:4x6_exotic}} and let $\rho_{kl}$ interchange rows $k$ and $l$.  Then the isotropy subgroup $H$ of $\bowtie$ is generated by the elements 
\[
\rho_{13}\rho_{24}\sigma_{12}, \ \ \rho_{12}\rho_{34}\sigma_{36}\sigma_{45},\ \ \sigma_{34}, \  \ \sigma_{56} .
\]
Also $\Fix(H)$ is not $\bowtie$, so $\bowtie$ is exotic.
\end{lemma}

\begin{proof}
The proof divides into four parts.

1.  The network divides into column components where two columns in a column component are either identical or can be made identical after color swapping. In this case there are two column components, one consisting of columns \{1,2\} and the other of columns \{3,4,5,6\}.  It is easy to see that elements in any subgroup $H \subseteq \ES_4\times\ES_6$ map column components to column-components; thus $H$ preserves column components because the two components contain different numbers of columns. Therefore $H$ is generated by elements of the form 
\[
(\rho,\alpha) \in \ES_4^{(1,2,3,4)} \times \ES_2^{(1,2)} \quad \mbox{or} \quad (\rho,\beta) \in \ES_4^{(1,2,3,4)} \times \ES_4^{(3,4,5,6)}
\]
where $\rho\in \ES_4^{(1,2,3,4)}$ permutes rows $\{1,2,3,4\}$, $\alpha \in\ES_2^{(1,2)}$ permutes columns $\{1, 2\}$, and  $\beta \in\ES_4^{(3,4,5,6)}$ permutes columns $\{3,4,5,6\}$.  

2. The only element of $H$ that contains the column swap $\sigma_{12}$ is 
$\rho_{13}\rho_{24}\sigma_{12}$, since $\sigma_{12}$ swaps colors in the first column component. This is the first generator in $H$.  

3. Elements in $H$ that fix colors in the second column component are generated by $\sigma_{34}, \; \sigma_{56}$.  The elements that swap colors are generated by $ \rho_{12}\rho_{34}\sigma_{36}\sigma_{45}$.

4. It is straightforward to verify that any element in $H$ fixes the pattern in Figure~\ref{F:4x6_exotic} (right).
\end{proof}

\subsection{$4\times k$ exotic dissensus value patterns}

In this section we show that the example of exotic value pattern in Section~\ref{S:EP} is not a rare case. Exotic $4\times k$ dissensus various patterns are indeed abundant.

\subsubsection{Classification of $4\times k$ 2-color Latin rectangles}

Let $B$ be a $4\times k$ Latin rectangle. With colors $R,B$ there are ten possible columns: four with one $R$ and six with two $R$s.

\begin{figure}[htb]
	\centering
	\includegraphics[width=.5\textwidth]{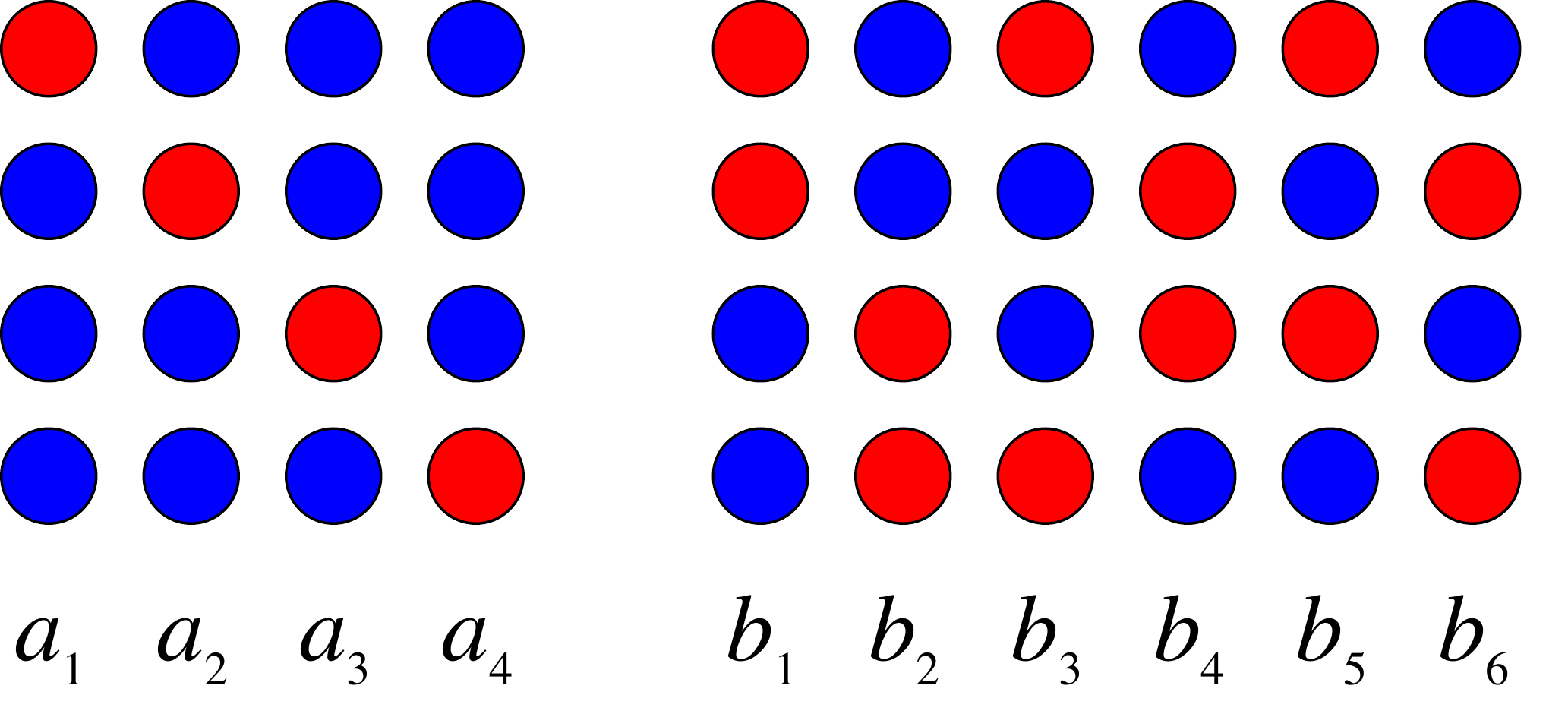}
	\caption{The ten $2$-color $4$-node columns divided into two sets.}
	\label{F:4xn_cols}
\end{figure}

Column balance and row balance imply that the columns are color-isomorphic, so either the columns are in the first set \hbox{$a_1$ -- $a_4$} or the second set $b_1$ -- $b_6$.

For the first set, row-balance implies that all four columns must
occur in equal numbers. Hence this type of Latin rectangle can only occur if the $k$ is a multiple of $4$. For $k=12$, the pattern looks like Figure~\ref{F:4xn_1R}. This pattern is always orbital.

\begin{figure}[htb]
	\centering
	\includegraphics[width=.55\textwidth]{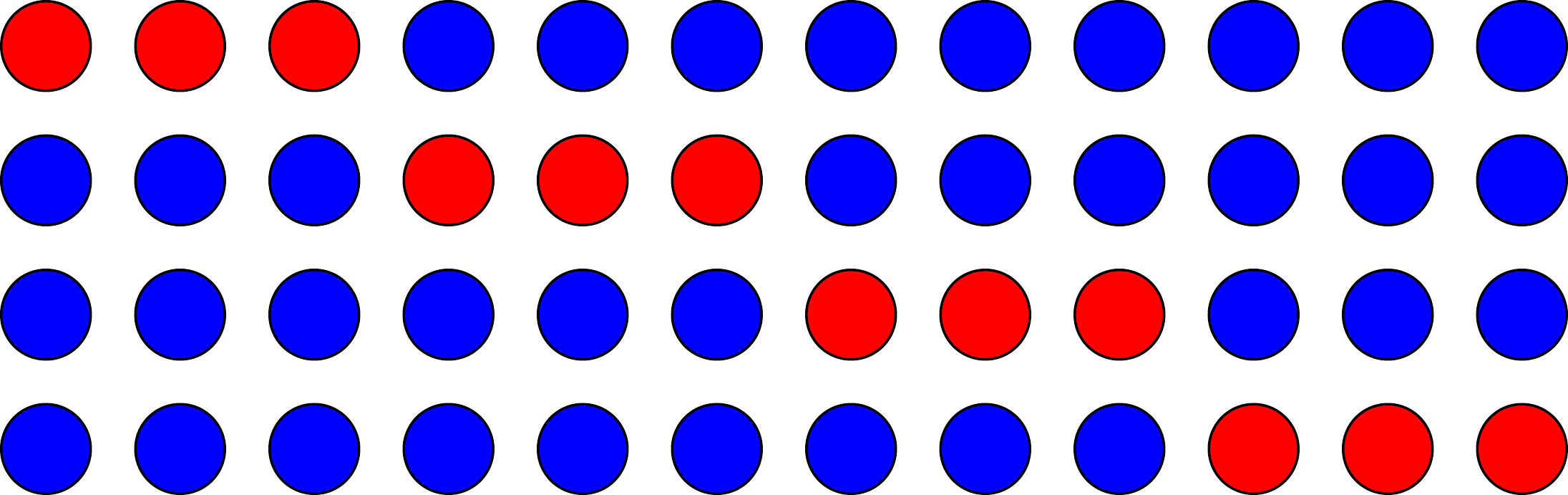}
	\caption{Typical pattern from first set of columns $a_1$ -- $a_4$ in Figure~\ref{F:4xn_cols}.}
	\label{F:4xn_1R}
\end{figure}

The second set is more interesting. We have grouped the column patterns in
color-complementary pairs: $(b_1,b_2)$,  $(b_3,b_4)$, and  $(b_5,b_6)$.

\begin{lemma}
\label{L:pair_lemma}
Suppose a $4\times n$ Latin rectangle has $\mu_i$ columns of type $b_i$.
Then
\begin{equation}
\label{E:equal_pairs}
\mu_1 = \mu_2=\nu_1 \qquad \mu_3 = \mu_4=\nu_2 \qquad \mu_5 = \mu_6 =\nu_3.
\end{equation}
\end{lemma}
\begin{proof}
Suppose the Latin rectangle has $\mu_i$ columns of type $b_i$. The row-balance
condition is that the total number of $R$ nodes in each row is the same. Therefore
the sums
\begin{align*}
	&\mu_1 + \mu_3 + \mu_5 \\
	&\mu_1 + \mu_4 + \mu_6 \\
	&\mu_2 + \mu_4 + \mu_5 \\
	&\mu_2 + \mu_3 + \mu_6 
\end{align*}
are all equal. This yields three independent equations in six unknowns.
It is not hard to see that this happens if and only if~\eqref{E:equal_pairs} holds.
\end{proof}

Lemma~\ref{L:pair_lemma} shows that in $4\times n$ Latin rectangles,
 color-complementary pairs of columns occur in equal numbers. For each color-complementary pair, we call the number $\nu_i$ its {\it multiplicity}. For purposes of illustration, we
can permute columns so that within the pattern identical columns occur in blocks. A typical pattern is Figure~\ref{F:4xn_2R}. Observe that this type of Latin rectangle exist only when $k$ is even.

\begin{figure}[htb]
	\centering
	\includegraphics[width=.55\textwidth]{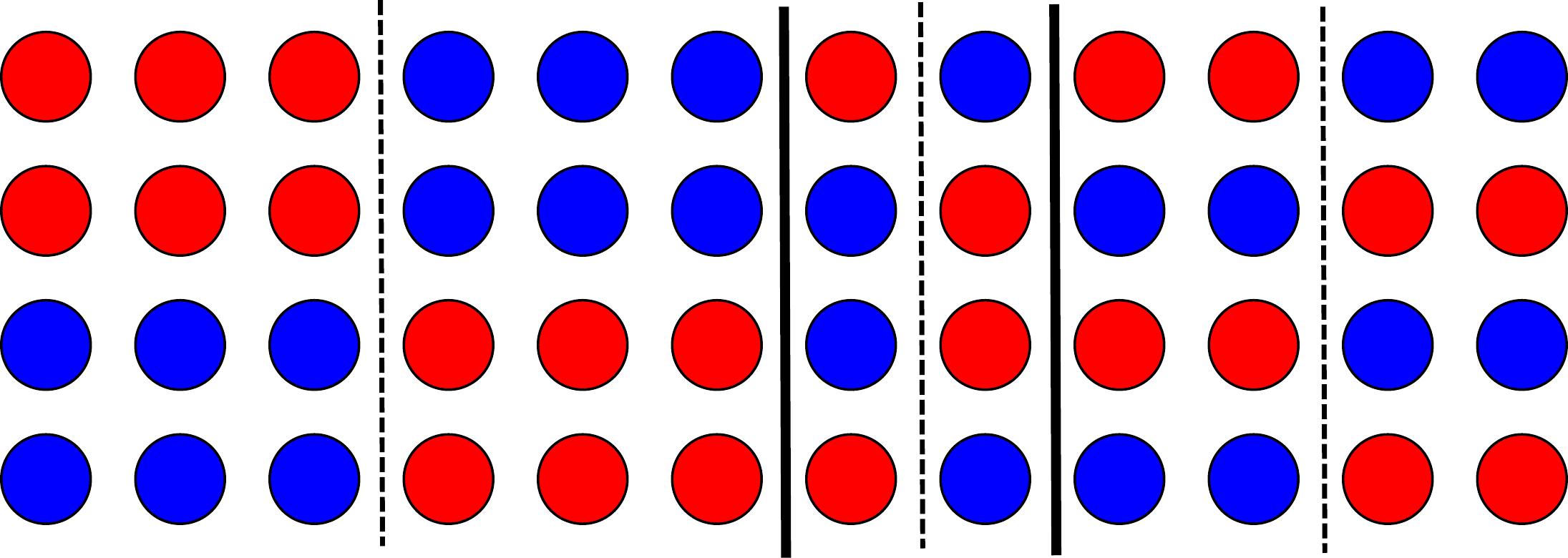}
	\caption{Typical pattern from second set of columns $b_1$ -- $b_6$ in Figure~\ref{F:4xn_cols}. The multiplicity of the color-complementary pair $(b_1,b_2)$ is 3. The multiplicity of the color-complementary pair $(b_3,b_4)$ is 1. The multiplicity of the color-complementary pair $(b_5,b_6)$ is 2.}
	\label{F:4xn_2R}
\end{figure}

\subsubsection{Sufficiency for a $\boldsymbol{4\times k}$ Latin rectangle with 2 colors to be exotic}

We now prove a simple sufficient condition for a $2$-coloring of a
$4\times n$ Latin rectangle to be exotic. In particular this gives another proof for the
$4\times 6$ example. 

\begin{theorem}
\label{T:exotic_suff}
Let $\bowtie$ be a $2$-coloring of a
$4\times n$ Latin rectangle in which the colors $R$ and $B$ occur in equal proportions.
Suppose that two color-complementary pairs have different multiplicities.
Then $\bowtie$ is exotic. 
\end{theorem}
\begin{proof}
Each column contains two $R$ nodes and two $B$ nodes. By Lemma~\ref{L:pair_lemma}
color-complementary pairs of columns occur in equal numbers, so the multiplicity of 
such a pair is defined.
The isotropy subgroup $H$ of $\bowtie$ acts transitively on the set
of nodes of any given color, say $R$. Let $C_1$ be one color-complementary pair
of columns and let $C_2$ be the other.

The action of $\ONE \times S_n \subseteq \ES_m \times \ES_n$
permutes the set of columns. It preserves the pattern in each column, so
it preserves complementary pairs.
The action of $\ES_m \times \ONE \subseteq \ES_m \times \ES_n$
leaves each column fixed setwise. It permutes the pattern in each column
simultaneously, so it maps complementary pairs of columns to complementary pairs.

Therefore any element of $\ES_m \times \ES_n$, and in particular
of $H$, preserves the multiplicities of complementary pairs. Thus
$H$ cannot map any column of $C_1$ to a column of $C_2$,
so $H$ cannot map any node in $C_1$ to a node in $C_2$.
But both $C_1$ and $C_2$ contain an $R$ node by the Latin rectangle property,
which contradicts $H$ being transitive on $R$ nodes.
\end{proof}

To construct exotic $4 \times n$ axial $2$-color patterns
for all even $n \geq 6$, we concatenate blocks of complementary pairs
of columns that do not all have the same multiplicity. The $4 \times 6$
pattern of Figure~\ref{F:4x6_exotic} (left) is an example.

This proof relies on special properties of columns of length $4$, notably
Lemma~\ref{L:pair_lemma}. Columns of odd length cannot occur as 
complementary pairs in Latin rectangles. However, at least one exotic pattern 
exists for a $5\times5$ influence network~\cite[Stewart]{S21}, and it seems 
likely that they are common for larger $m,n$. More general sufficiency conditions 
than that in Theorem~\ref{T:exotic_suff} can probably be stated.

\section{An $\mathcal N_{mn}$-admissible value-formation ODE}

We propose the following $\mathcal N_{mn}$-admissible value-formation dynamics:
\begin{equation}\label{eq: admissible value formation}
\dot z_{ij} = -z_{ij} + \lambda \left( S_1 \left({\talpha z_{ij} + \sum_{k\not= i} \tgamma z_{kj}}\right) 
	+ \sum_{l\not= j} S_2 \left( \tbeta z_{il} + \sum_{k\not = i} \tdelta z_{kl} \right) \right) 
\end{equation}
Here $S_1$ and $S_2$ are sigmoidal functions, $1 \leq i \leq m$, and $1\leq j \leq n$. Model~\eqref{eq: admissible value formation} is intimately related to and inspired by the model of opinion dynamics introduced in~\cite[Bizyaeva {\it et al.}]{BFL22}.

Model~\eqref{eq: admissible value formation} has two main components: a linear degradation term (modeling resistance to change assigned values) and a saturated network interaction term whose strength is controlled by the bifurcation parameter $\lambda$. The network interaction term is in turn composed of four terms, modeling the four arrow types of  $\mathcal N_{mn}$ networks. The sigmoids $S_1$ and $S_2$ satisfy
\begin{equation}\label{eq:sigmoid conditions}
	S_i(0)=0,\quad S_i'(0)=1,\quad S_i^{(n)}(0)\neq 0,\ n\geq 2.
\end{equation}
With this choice, the trivial equilibrium for this model is $z_{ij}=0$ for all $i,j$.

In simulations, we use
\begin{equation}\label{eq: sigmoid definition}
	S_i(x)=\frac{\tanh(x-s_i)+\tanh(s_i)}{1-\tanh(s_i)^2},
\end{equation}
which satisfies~\eqref{eq:sigmoid conditions} whenever $s_i\neq 0$.

\subsection{Parameter interpretation}

 Parameter $\talpha$ tunes the weight of the node self-arrow, that is, the node's internal dynamics. When $\alpha<0$ resistance to change assigned value is increased through nonlinear self-negative feedback on option values. When $\alpha>0$ resistance to change assigned value is decreased through nonlinear self-positive feedback on option values.

 Parameter $\tbeta$ tunes the weight of row arrows, that is, intra-agent, inter-option interactions. When $\tbeta<0$, an increase in one agent's valuation of any given option tends to decrease its valuations of other options. When $\tbeta>0$, an increase in one agent's valuation of any given option tends to increase its valuations of other options.

Parameter $\tgamma$ tunes the weight of column arrows, that is, inter-agent, intra-option interactions. When $\tgamma<0$, an increase in one agent's valuation of any given option tends to decrease other agents' valuations of the same option. When $\tgamma>0$, an increase in one agent's valuation of any given option tends to increase other agents' valuations of the same option.

Parameter $\tdelta$ tunes the weight of diagonal arrows, that is, inter-agent, inter-option interactions. When $\tdelta<0$, an increase in one agent's valuation of any given option tends to decrease other agents' valuations of the all other options. When $\tdelta>0$, an increase in one agent's valuation of any given option tends to increase other agents' valuations of all other options.

\subsection{Bifurcation conditions}

Conditions for the various types of synchrony-breaking bifurcation can be computed by noticing that in model~\eqref{eq: admissible value formation}
\begin{equation}\label{EQ: model deadlocked jac eigvals}
\cd =  -1+\lambda\tcd \quad \cc =  -1 + \lambda\tcc\quad
\cdd =  -1 +\lambda\tcdd \quad \cdl =  -1 + \lambda\tcdl
\end{equation}
where
\begin{equation*}
	\begin{aligned}
		\tcd & =  \talpha-\tbeta-\tgamma+\tdelta \\
		\tcc & =   \talpha-\tbeta+(m-1)\left(\tgamma-\tdelta\right)\\
		\tcdd & =  \talpha-\tgamma+(n-1)\left(\tbeta-\tdelta\right)\\
		\tcdl & =  \talpha+(n-1)\tbeta+(m-1)\tgamma+(m-1)(n-1)\tdelta
	\end{aligned}
\end{equation*}

For instance, if $\tcd>0$ and $\tcd>\tcc,\tcdd,\tcdl$, then a dissensus  synchrony-breaking bifurcation happens at $\lambda=\frac{1}{\tcd}$.

\section{Stable solutions}
\label{S:stable_solutions}

This section divides into four parts:
\begin{itemize}

\item[\eqref{S:simulation}] Simulation of consensus and deadlock synchrony-breaking.

\item[{\eqref{S:simulation_dissensus}}]  Simulation of dissensus synchrony-breaking.

\item[\eqref{SS:stability_dissensus}] Discussion of the stability of dissensus bifurcation branches.

\item[\eqref{SS:stability_balanced}]  Stable equilibria can exist in balanced colorings.

\end{itemize}

\subsection{Simulation of consensus and deadlock synchrony-breaking}
\label{S:simulation}

To simulate consensus and deadlock synchrony-breaking we use $s_1=0.5$ and  $s_2=0.3$ 
in~\eqref{eq: admissible value formation} and \eqref{eq: sigmoid definition}. For consensus 
synchrony-breaking we set
\begin{equation} \label{e:consensus_bifurcation}
		\tcd =  \tcdd = \tcdl -1.0,\quad  \tcc = 1.0, \quad  \lambda = -1 +\varepsilon;
\end{equation}
and for deadlock synchrony-breaking we set
\begin{equation}
\label{e:deadlock_bifurcation}
		\tcd = \tcc  = \tcdl =  -1.0,  \quad\tcdd =  1.0, \quad\lambda = 1+\varepsilon, 
\end{equation}
where $0 \leq \varepsilon\ll1$. In other words, both for consensus and deadlock synchrony-breaking we let the bifurcation parameter $\lambda$ be slightly above the critical value
at which bifurcation occurs.  In simulations, we use $\varepsilon=10^{-2}$. Initial conditions are chosen 
randomly in a small neighborhood of the unstable fully synchronous equilibrium.

\begin{figure}
	\centering
	\begin{subfigure}[b]{0.8\textwidth}
		\centering
		\includegraphics[width=\textwidth]{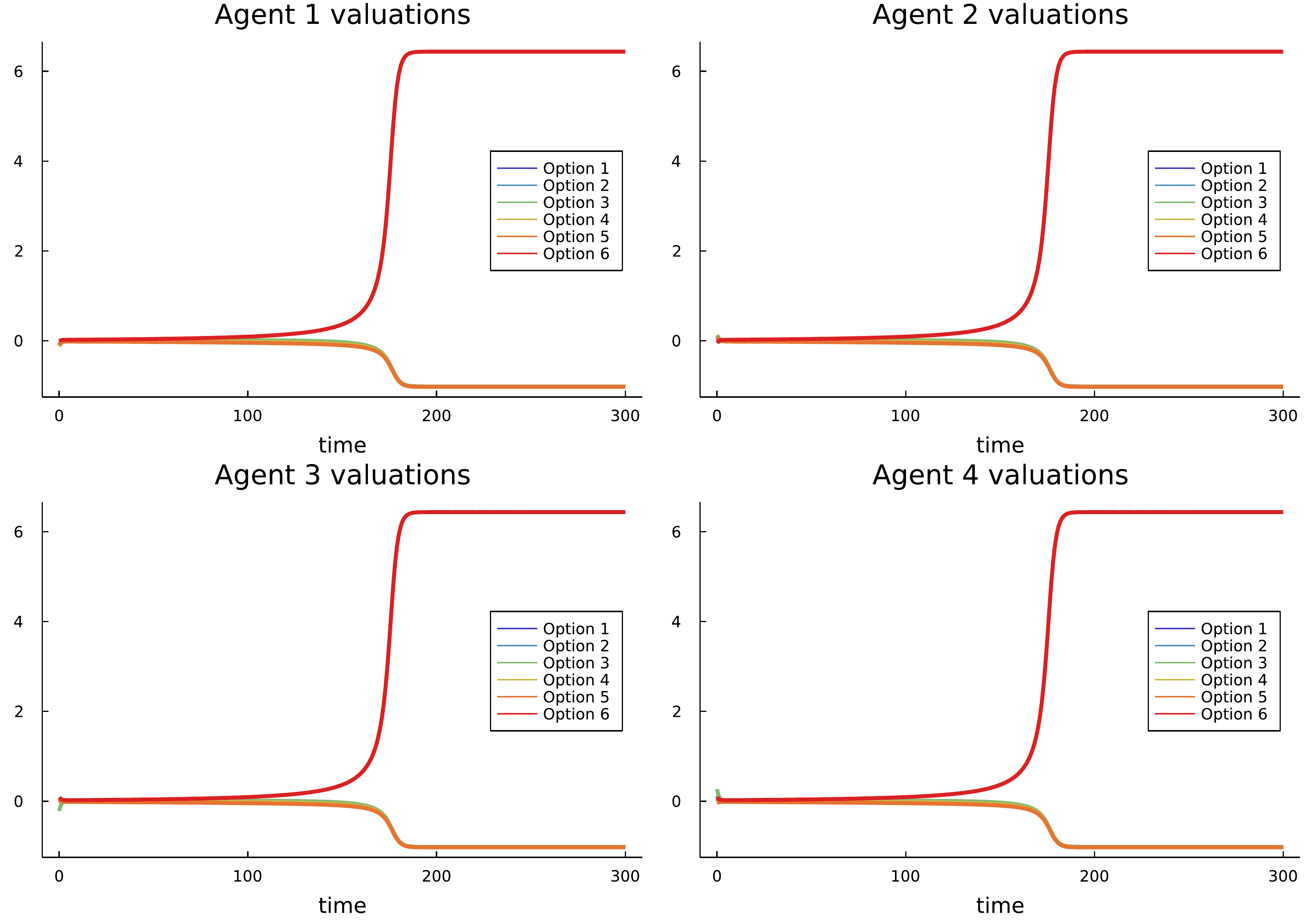}
		\caption{Evolution of valuations by 4 agents on 6 options at consensus symmetry-breaking. Simulation 
		of \eqref{eq: admissible value formation} with parameters~\eqref{e:consensus_bifurcation}.
		}
		\label{fig:4 6 cons time}
	\end{subfigure}
	\hfill
	\begin{subfigure}[b]{0.5\textwidth}
		\centering
		\includegraphics[width=\textwidth]{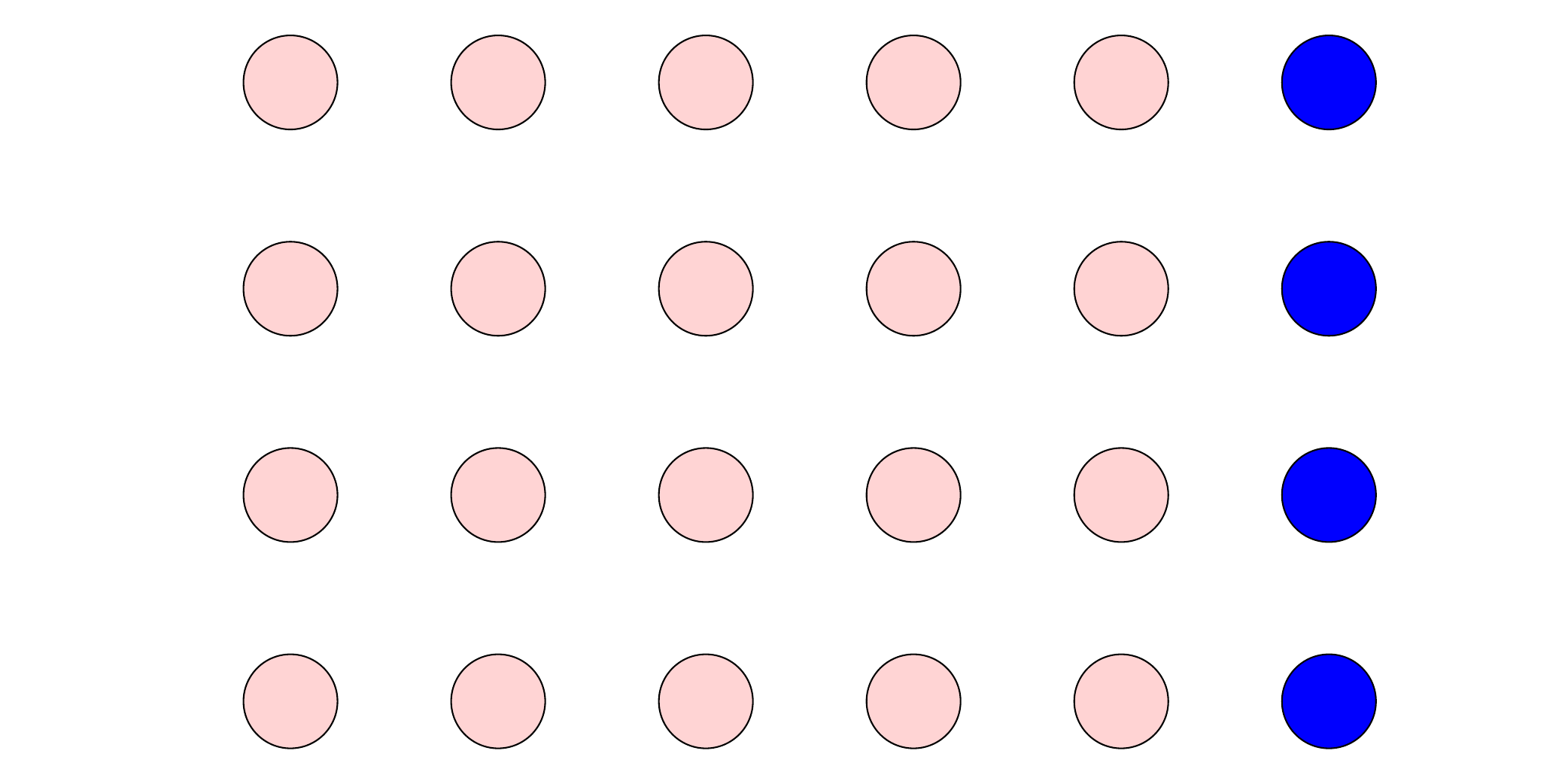}
		\caption{Final simulation pattern of valuations of 4 agents on 6 options at consensus symmetry-breaking. 
		Option 6 is chosen.}
		\label{fig:4 6 cons patt}
	\end{subfigure}
\caption{Possible consensus symmetry-breaking between 4 agents and 6 options. Transient value-assignment is shown through time series in (a). The final value pattern is shown in (b) and is of consensus type.}
\label{fig:4 6 cons}
\end{figure}

\begin{figure}
	\centering
	\begin{subfigure}[b]{0.8\textwidth}
		\centering
		\includegraphics[width=\textwidth]{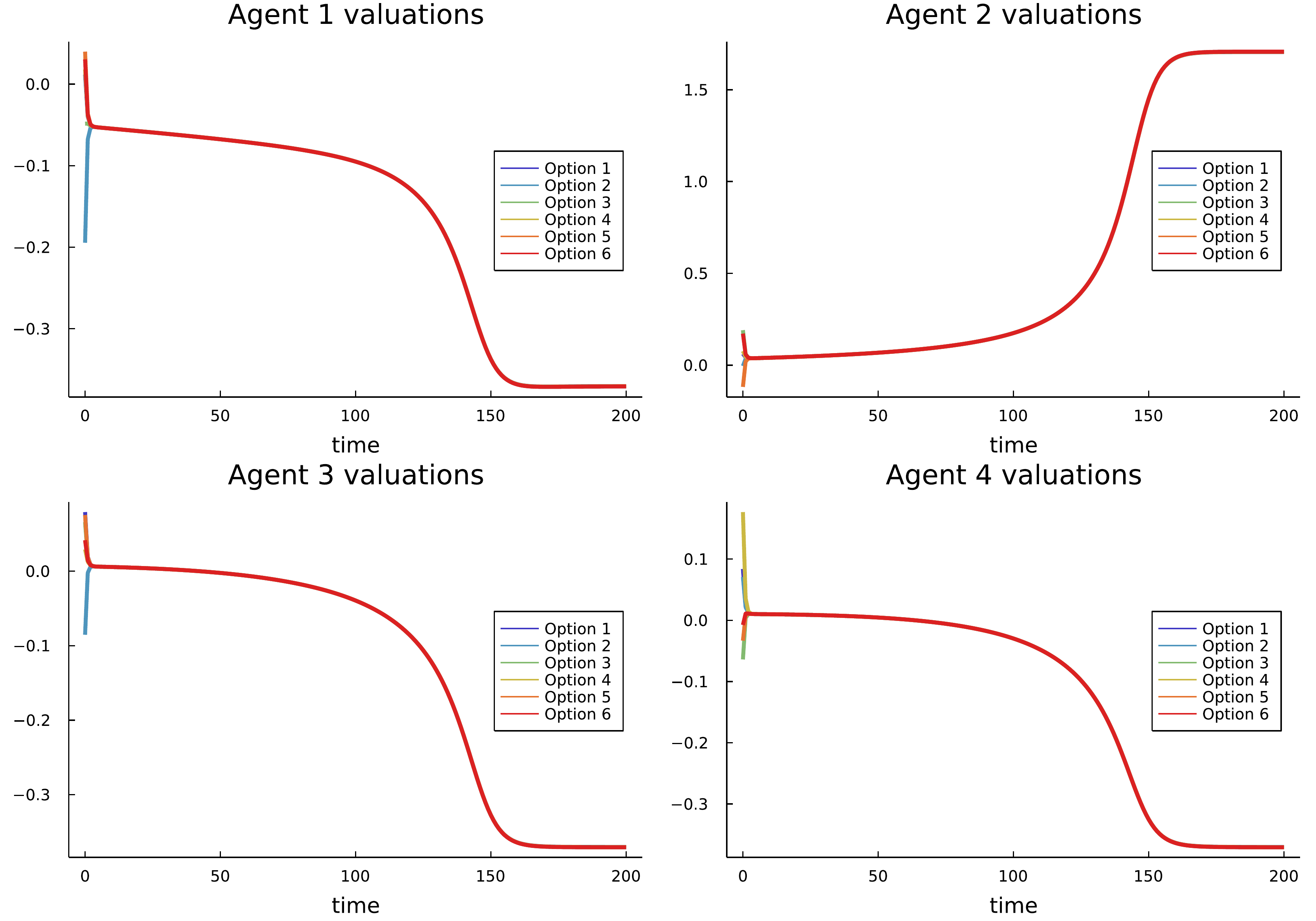}
		\caption{Evolution of valuation by 4 agents on 6 options at deadlock symmetry-breaking. Simulation 
		of \eqref{eq: admissible value formation} with parameters~\eqref{e:deadlock_bifurcation}.}
		\label{fig:4 6 deadl time}
	\end{subfigure}
	\begin{subfigure}[b]{0.5\textwidth}
		\centering
		\includegraphics[width=\textwidth]{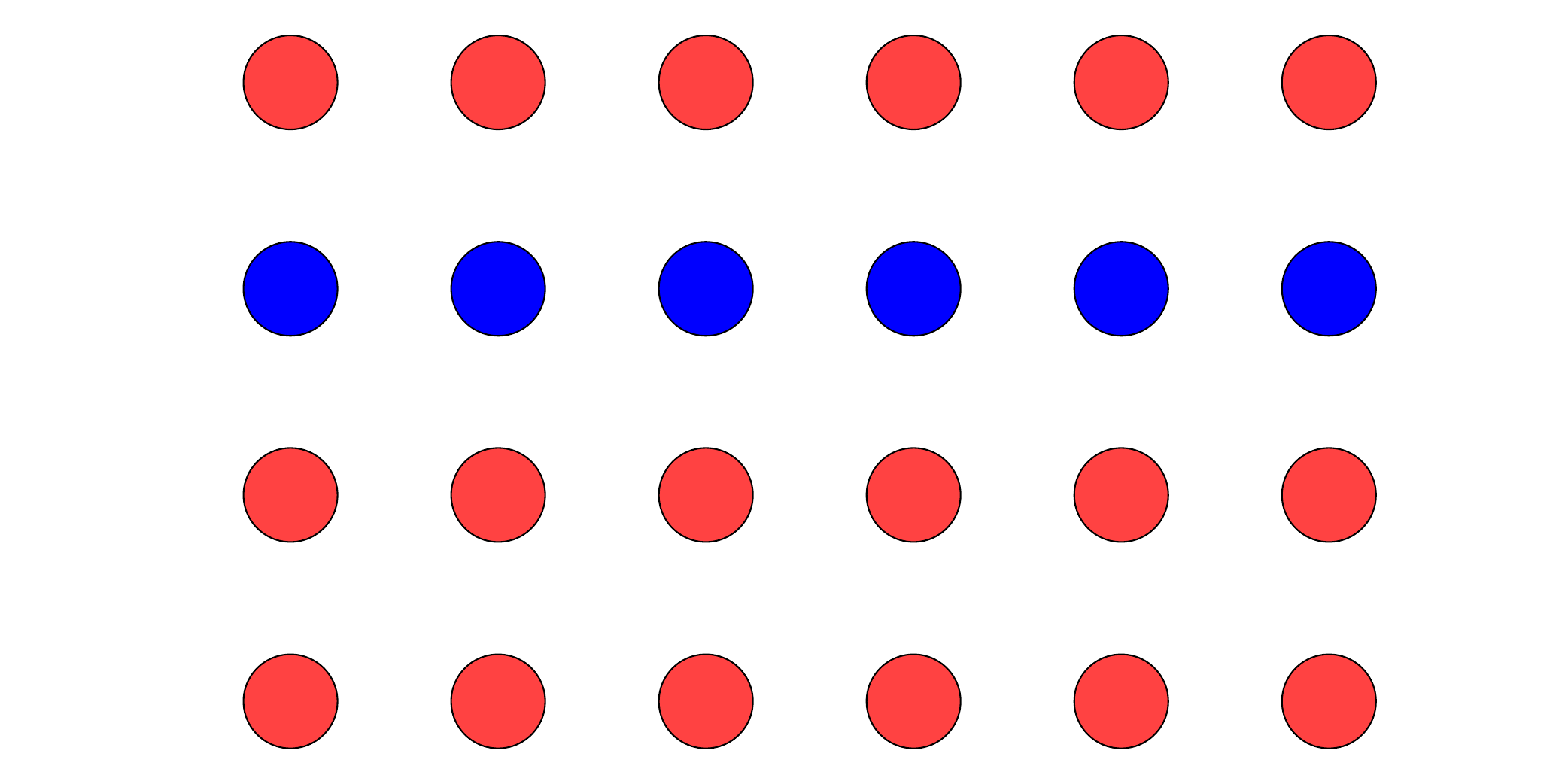}
		\caption{Final simulated value pattern by 4 agents on 6 options at deadlock symmetry-breaking. Agent 2 values all the options positively. Other agents value all the options negatively.}
		\label{fig:4 6 deadl patt}
	\end{subfigure}
	\caption{Possible deadlock symmetry-breaking pattern obtained by simulation of 4 agents on 6 options. Transient value-assignment is shown through time series in (a). The final value pattern is shown in (b) and is of deadlock type.}
	\label{fig:4 6 deadl}
\end{figure}

After exponential divergence from the neutral point (Figures~\ref{fig:4 6 cons time} and~\ref{fig:4 6 deadl time}), trajectories converge to a consensus (Figure~\ref{fig:4 6 cons patt}) or deadlock (Figure~\ref{fig:4 6 deadl patt}) value pattern, depending on the chosen bifurcation type. Observe that the value patterns trajectories converge to are `far' from the neutral point. In other words, the network `jumps' away from indecision to a value pattern with distinctly different value assignments compared to those before bifurcation. This is a consequence of model-independent stability properties of consensus and deadlock branches, as we discuss next. A qualitatively similar behavior would have been observed for $\varepsilon=0$, i.e., for $\lambda$ exactly at bifurcation, with the difference that divergence from the fully synchronous equilibrium would have been sub-exponential because of zero (instead of positive) eigenvalues of the linearization. Also, for $\varepsilon=0$, trajectories could have converged to different consensus or deadlock value patterns as compared to $\varepsilon=10^{-2}$ because of possible secondary bifurcations that are known to happen close to $\ES_n$-equivariant bifurcations and that lead to primary-branch switching~~\cite[Cohen and Stewart]{CS00}, \cite[Stewart {\it et al.}]{SEC03},\cite[Elmhirst]{E02}.

\subsection{Simulation of dissensus synchrony-breaking}
\label{S:simulation_dissensus}

We simulate dissensus synchrony-breaking for two sets of parameters:
\begin{eqnarray}
 \label{param1}
&&\qquad \tcd = 1.0,  \tcc = -1.0,  \tcdd = -0.5, \tcdl = -0.5, s_1 = 0.5,  s_2 = 0.3, \lambda = \frac{1}{\tcd}+\varepsilon \\
\label{param2}
&&\qquad \tcd = 1.0,  \tcc = -1.0,  \tcdd = -1.0, \tcdl = -1.0,  s_1 = -0.1, s_2 = -0.3, \lambda = \frac{1}{\tcd}+\varepsilon 
\end{eqnarray}
Initial conditions are chosen randomly in a small neighborhood of the unstable neutral point, and in simulations $\varepsilon=10^{-2}$. The resulting temporal behaviors and final value patterns are shown in 
Figures~\ref{fig:4 6 orb diss} and \ref{fig:4 6 exo diss}.

\begin{figure}
	\centering
	\begin{subfigure}[b]{0.8\textwidth}
		\centering
		\includegraphics[width=\textwidth]{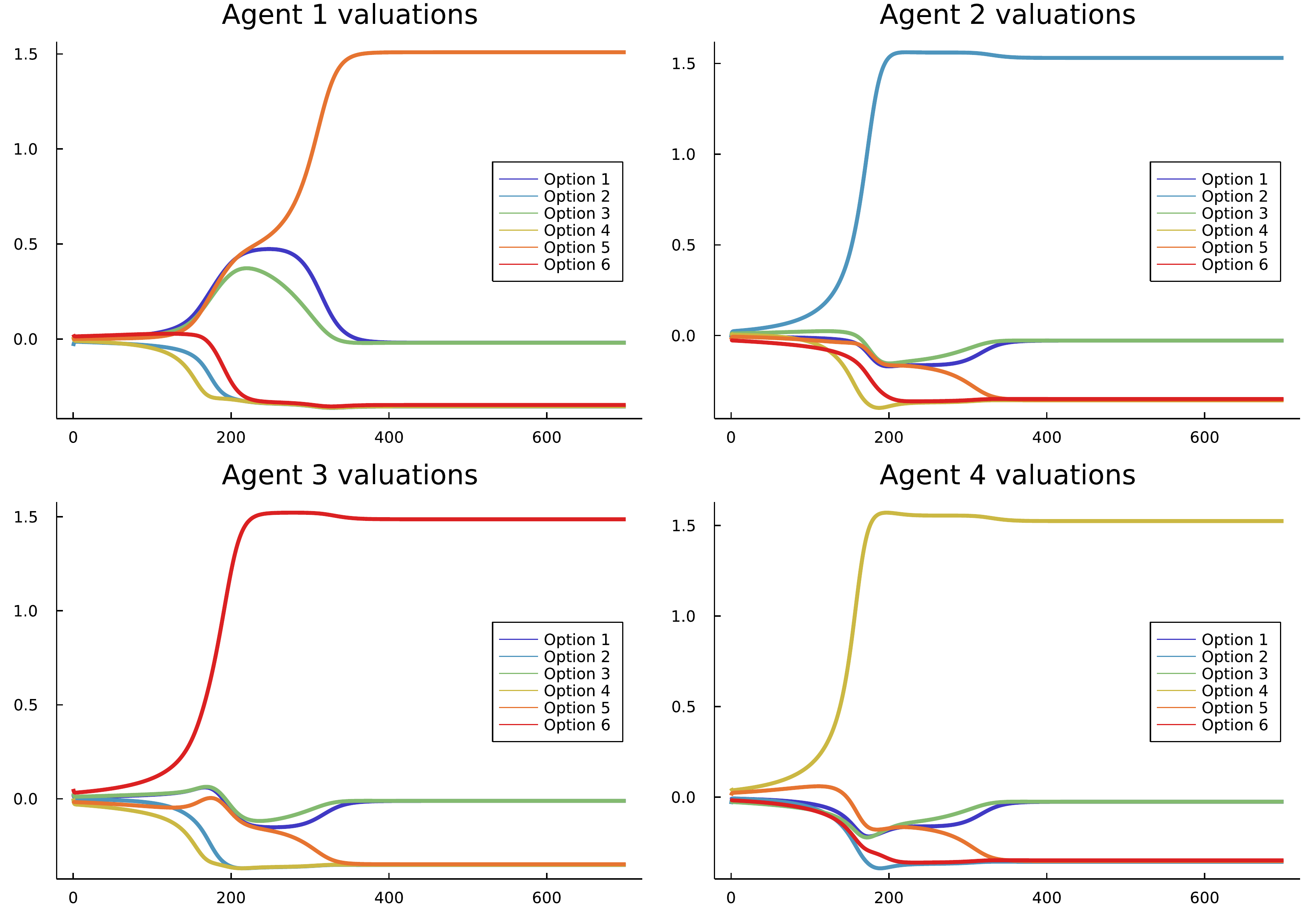}
		\caption{Evolution of 4 agents' valuations about 6 options at dissensus symmetry-breaking with 
		parameter set \eqref{param1}.}
		\label{fig:4 6 orb diss time}
	\end{subfigure}
	\hfill
	\begin{subfigure}[b]{0.5\textwidth}
		\centering
		\includegraphics[width=\textwidth]{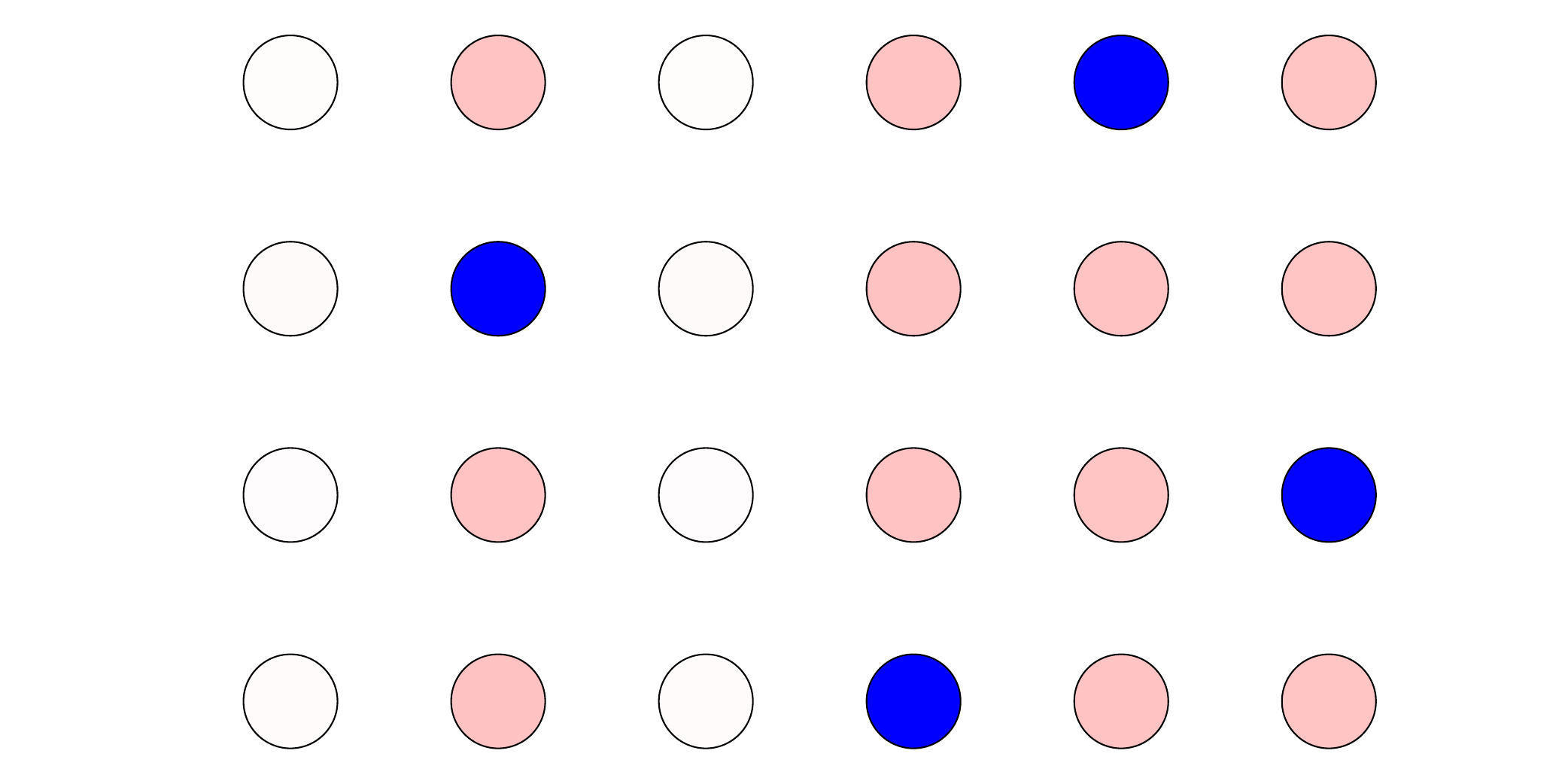}
		\caption{Final value pattern of 4 agents' valuations about 6 options at dissensus symmetry-breaking with parameter set \eqref{param1}. All agents are neutral about Options~1 and~3. Agent~1 favors Option~5. Agent~2 favors Option~2. Agent~3 favors Option~6. Agent~4 favors Option~4.}
		\label{fig:4 6 orb diss patt}
	\end{subfigure}
	\caption{Dissensus symmetry-breaking between 4 agents and 6 options  with parameter set  \eqref{param1}. Transient value-assignment is shown through time series in (a). The final value pattern is shown in (b) and is of orbital dissensus type.}
	\label{fig:4 6 orb diss}
\end{figure}

\begin{figure}
	\centering
	\begin{subfigure}[b]{0.8\textwidth}
		\centering
		\includegraphics[width=\textwidth]{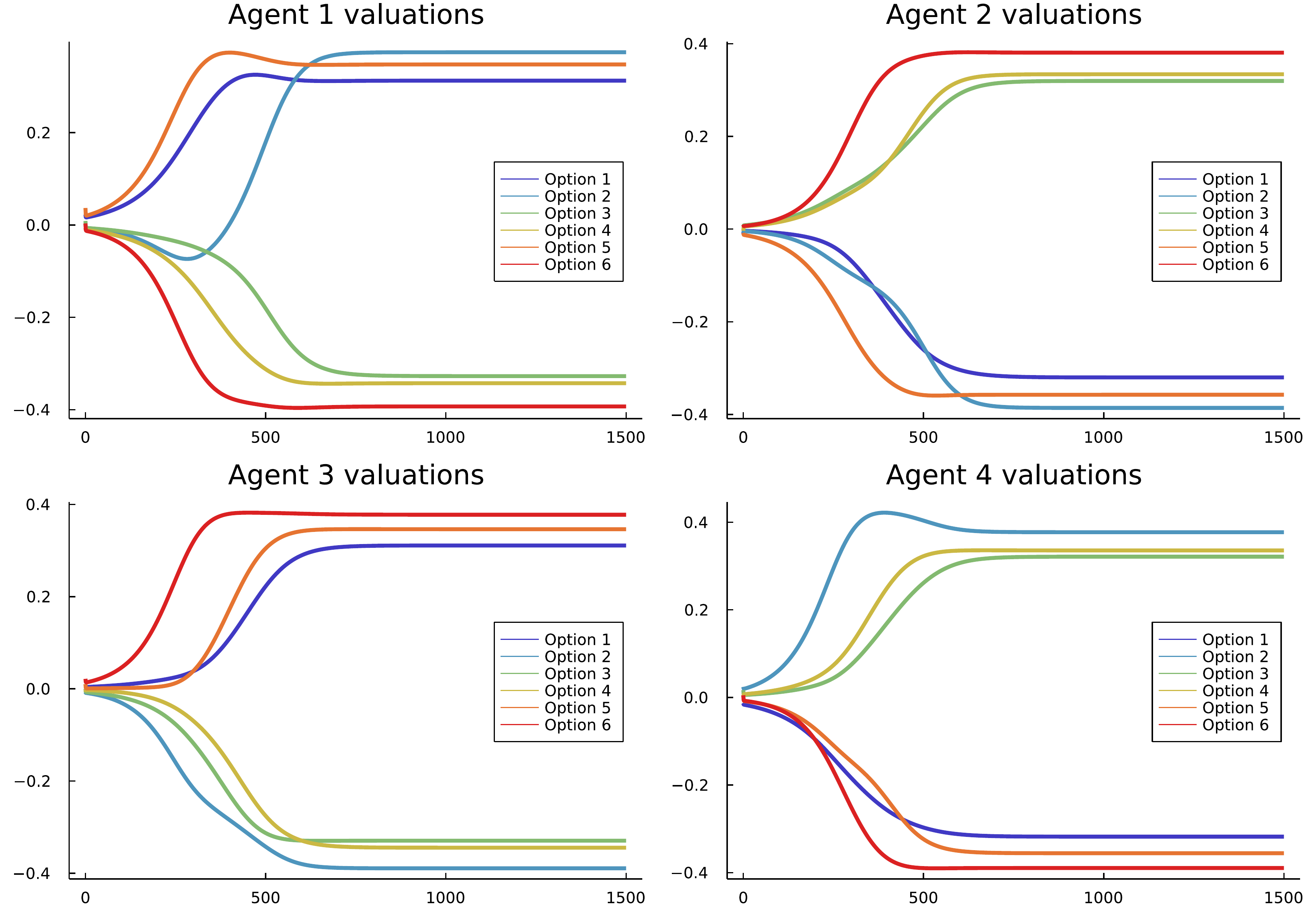}
		\caption{Evolution of 4 agents' valuations about 6 options at dissensus symmetry-breaking with 
		parameter set  \eqref{param2}.}
		\label{fig:4 6 exo diss time}
	\end{subfigure}
	\hfill
	\begin{subfigure}[b]{0.5\textwidth}
		\centering
		\includegraphics[width=\textwidth]{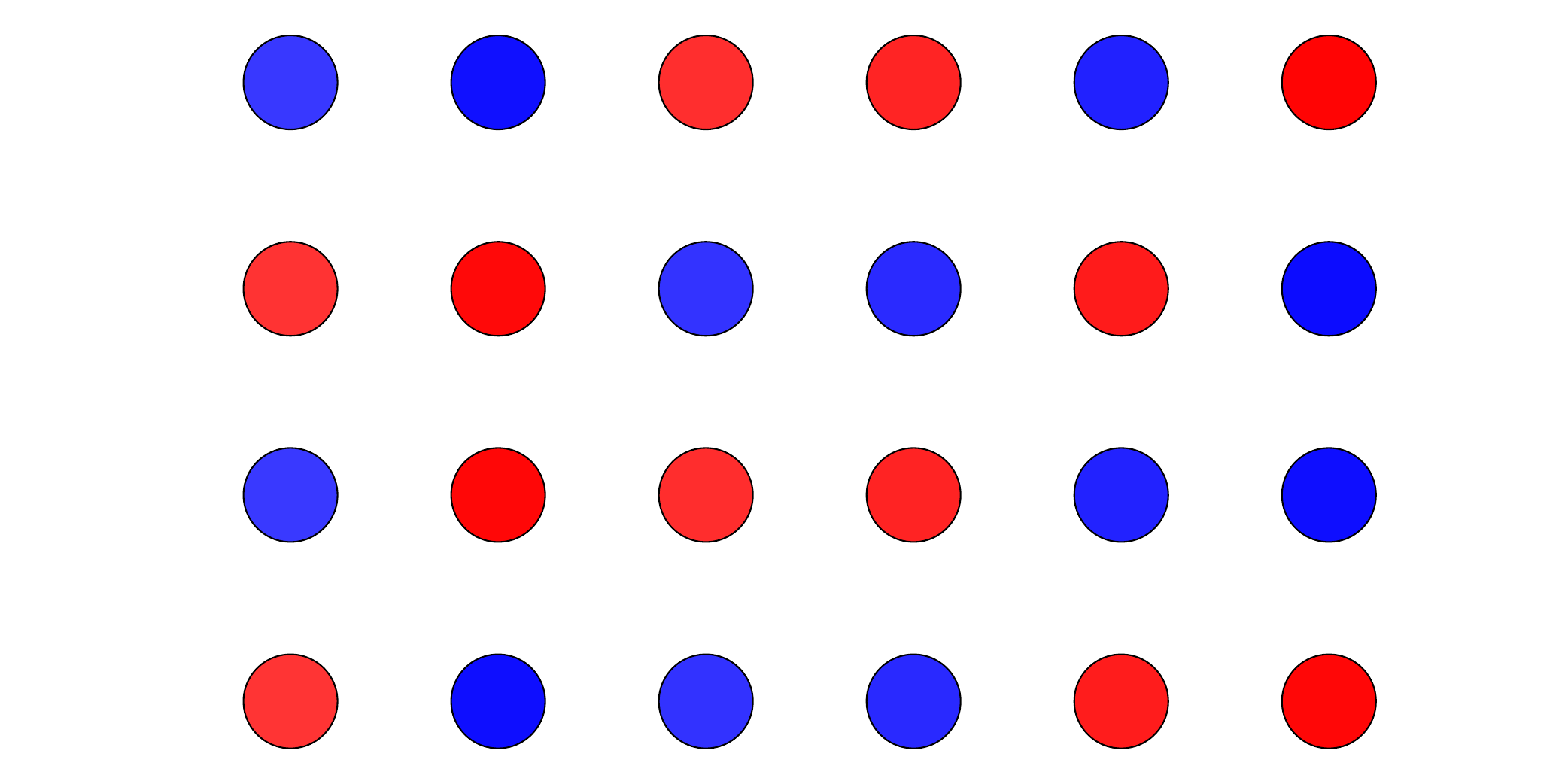}
		\caption{Final value pattern of 4 agents' valuations about 6 options at dissensus symmetry-breaking with parameter set \eqref{param2}. Agent~1 favors Options~1,2,5. Agent~2 favors Options~3,4,6. Agent~3 favors Options~1,5,6. Agent~4 favors Options~2,3,4.}
		\label{fig:4 6 exo diss patt}
	\end{subfigure}
	\caption{Dissensus symmetry-breaking between 4 agents and 6 options  with parameter set \eqref{param2}. Transient value-assignment is shown through time series in (a). The final value pattern is shown in (b) and is of exotic dissensus type.}
	\label{fig:4 6 exo diss}
\end{figure}

For both sets of parameters, trajectories jump toward a dissensus value pattern. For the first set of parameters (Figure~\ref{fig:4 6 orb diss}), the value pattern has a block of zeros (first and second columns) corresponding to options about which the agents remain neutral. The pattern in Figure~\ref{fig:4 6 orb diss patt} is orbital, with symmetry group $\Z_4$. For the second set of parameters (Figure~\ref{fig:4 6 exo diss}), each agent favors half the options and dislikes the other half, but there is disagreement about the favored options. The precise value pattern follows the rules of a Latin rectangle with two red nodes and two blue nodes per column and three red nodes and three blue nodes per row. The pattern in Figure~\ref{fig:4 6 exo diss patt} is exotic, and is a permutation of the pattern in Figure~\ref{F:4x6_exotic} (left). 

With parameter set \eqref{param2},
model~\eqref{eq: admissible value formation} has stable equilibria with synchrony patterns 
given by both exotic and orbital dissensus value patterns.
It must be stressed that here, for the same set of parameters but different initial conditions trajectories converge to value patterns corresponding to different Latin rectangles with the same column and row proportions of red and blue nodes, both exotic and orbital ones.
Therefore multiple (conjugacy classes of) stable states can coexist, even
when proportions of colors are specified.

\subsection{Stability of dissensus bifurcation branches}
\label{SS:stability_dissensus}

As discussed in Section~\ref{S:SCB} for consensus and deadlock bifurcations, 
solutions given by the Equivariant Branching Lemma are often unstable near bifurcation, but 
can regain stability away from bifurcation. The way this happens in dissensus bifurcations is 
likely to be similar to the way $\ES_N$ axial branches regain stability, that is, through saddle-node 
`turning-around' and secondary bifurcations. This observation is verified by simulation for 
both orbital and exotic axial solutions, Section~\ref{S:simulation_dissensus}. The 
analytic determination of stability or instability for dissensus axial solutions 
involves large numbers of parameters and is beyond our capability.
 
We know, by the existence of a non-zero quadratic equivariant~\cite[Section~6.1]{FGBL20}, 
that Ihrig's Theorem might apply to orbital dissensus bifurcation branches. We do not know 
if a similar result applies to exotic orbital dissensus bifurcation branches but our numerical 
simulations suggest that this is the case. Our simulations also show that both types of 
pattern can be stable in our model for values of the bifurcation parameter close to the 
bifurcation point.

\subsection{Stable equilibria can exist for any balanced coloring}
\label{SS:stability_balanced}

Furthermore, we show that whenever $\bowtie$ is a balanced
coloring of a network, there exists an admissible ODE having a linearly stable
equilibrium with synchrony pattern $\bowtie$. 

\begin{theorem}
	\label{T:BCequi_exist}
	Let $\bowtie$ be a balanced coloring of a network $\GG$. Then for any
	choice of node spaces there exists a $\GG$-admissible map $f$ such
	that the ODE $\dot x = f(x)$ has a linearly stable equilibrium 
	with synchrony pattern $\bowtie$.
\end{theorem}
\begin{proof}
	Let $y$ be a generic point for $\bowtie$; that is, $y_c = y_d$ if and only if $c \bowtie d$.
	Balanced colorings refine input equivalence. Therefore
	each input equivalence class $\KK$ of nodes is a disjoint union of
	$\bowtie$-equivalence classes: $\KK = \KK_1 \dot\cup \cdots \dot\cup \KK_s$.
	If $c,d \in \KK$ then $y_c = y_d$ if and only if $c,d$ belong to the same $\KK_i$.
	Writing variables in standard order (that is, in successive blocks
	according to arrow-type) we may assume that $P_c = P_d$
	whenever $c \sim_I d$.
	
	Next, we define a map $f^\KK: P_\KK \to P_\KK$ such that
	\beqn
	f^\KK(y_d) &=& 0\quad \forall d \in \KK \\
	\mathrm{D}f^\KK(y_d) &=& -\id_d\quad  \forall d \in \KK
	\eeqn
	where $\id_d$ is the identity map on $P_d$.
	This can be achieved with a polynomial map by polynomial interpolation.
Now define $g:P\to P$ by
	\[
	g_c(x) = f^\KK(x_c) \quad \mbox{when}\ c \in \KK
	\]
	The map $g$ is admissible since it depends only on the node variable and its components on any 
	input equivalence class are identical. Now
	$g_c(y_d) = f^\KK(y_d) = 0$,
	so $y$ is an equilibrium of $g$. Moreover, $\mathrm{D}f|_y$ is a block-diagonal
	matrix with blocks $-\id_c$ for each $c \in \CC$; that is, 
	$\mathrm{D}f|_y = - \id_P$,
	where $\id_P$ is the identity map on $P$. The eigenvalues of $\mathrm{D}f|_y$
	are therefore all equal to $-1$, so the equilibrium $y$ is linearly stable.
\end{proof}

\end{document}